\documentclass{amsart}

\usepackage{amssymb, amsmath, latexsym, a4}
\usepackage[latin1]{inputenc}
\usepackage[all]{xy}
\usepackage{enumitem}

\def\1{^{-1}}
\def\id{{\sf id}}

\newtheorem{De}{Definition}

\newtheorem{Th}[De]{Theorem}
\newtheorem{Pro}[De]{Proposition}
\newtheorem{Le}[De]{Lemma}
\newtheorem{Co}[De]{Corollary}

\newtheorem{Rem}[De]{Remark}

\newtheorem*{Ex*}{Examples}
\newtheorem*{Example*}{Example}

\def\K{{\mathtt K}}
\def\G{{\mathtt G}}

\def\cat{{\sf Cat}}

\def\Der{{\sf Der}}

\def\ca{{\mathcal C }}

\def\H{{\mathtt H}}

\def\nc{\mathop{\bf Z}_*^{Nor}\nolimits}

\def\xto#1{\xrightarrow[]{#1}}
\def\d{\partial}
\def\1{^{-1}}
\def\im{{\sf Im}}
\def\ker{{\sf Ker}}
\def\cok{\sf Coker}

\def\al{{\alpha}}

\def\dd{{\delta}}

\def\t{\otimes}

\begin{document}

	\title{On the Gottlieb group, Drinfeld centre and the centre of a crossed module}
	\author{Mariam Pirashvili}
		
	\maketitle
	
\begin{abstract}
	The aim of this paper is to introduce the concept of the centre of a crossed module \(\G_* = (\G_2\to \G_1)\). This centre is closely related to the Gottlieb group of the classifying space of a crossed module and also to the Drinfeld centre of a monoidal category introduced independently by Drinfeld and Joyal and Street. Our definition of the centre is based on certain crossed homomorphisms $\G_1\to \G_2$, which makes it easy to relate it to group cohomology. 
\end{abstract}

\section{Introduction}

The aim of this work is to introduce and prove essential properties of the centre of a crossed module. In particular, our results establish a connection between the Drinfeld-Joyal-Street centre of a monoidal category \cite{jst} and the Gottlieb group \cite{gottlieb1} of a pointed topological space. To the best of our knowledge, such a link between these classical objects has not been observed before.

For a CW-complex $X$ we define the \emph{centre} of $X$ to be the connected component of the mapping space $Map(X,X)$ containing the identity map $\id_X$. Here $Map(X,X)$ is the set of all continuous maps $X\to X$, equipped with the compact-open topology. The centre of $X$ will be denoted by ${\mathcal Z}X$. It  has  an $H$-space structure induced by the composition of maps. In particular, the group $\pi_1({\mathcal Z}X,\id_X)$ is abelian. These spaces are classical objects, although not known by this name. They have been extensively studied, see for example the survey paper \cite{survey} and the references given therein. We are interested in these spaces in the following context in which they arose in Gottlieb's work \cite{gottlieb1}.

Let $(X,x_0)$ be a pointed space. Then the evaluation at $x_0$ defines the pointed map $ev_{x_0}:({\mathcal Z}X,\id_X)\to  (X,x_0),$ which induces a group homomorphism $$\pi_1({\mathcal Z}X,\id_X)\to \pi_1(X,x_0).$$
The image of this homomorphism is denoted by $G(X,x_0)$ and is called the \emph{Gottlieb group} of $(X,x_0)$.  
Gottlieb, among other results, proved that if $X=B\pi$ is the classifying space of a discrete group $\pi$, then there is a homotopy equivalence
$$ {\mathcal Z}(B\pi )\simeq B({\mathcal Z}\pi)
$$ 
and an isomorphism of groups $$G(B\pi,1)\cong {\mathcal Z}\pi,$$
 where ${\mathcal Z}\pi$ is the centre of the group $\pi$, see \cite[Theorem III.2]{gottlieb1} and \cite[Corollary 1.13 ]{gottlieb1}.

Our aim is to extend these results to spaces $X$ for which $\pi_i(X)=0$ for $i\not = 1,2$. It is well-known that algebraic models for such spaces are crossed modules, see for example \cite{jll}, \cite{4dim},  \cite{nat}.  
In fact, any crossed module ${\G}_*=(\G_2\to \G_1)$ has a classifying space $B{\G}_*$, which is connected and has vanishing homotopy groups in dimensions $\geq 3$. Conversely, any CW-complex of such type is homotopy equivalent to $B{\G}_*$ for some crossed module ${\G}_*$. Moreover, one can assume that $\G_1$ is a free group.

By definition, a crossed module $\G_*$ is a group homomorphism $\d:\G_2\to \G_1$ together with an action of $\G_1$ on $\G_2$ satisfying some properties (see Section \ref{13se}). The most important invariants of the crossed module $\G_*$ are the group $\pi_1(\G_*)=\cok(\d)$ and the $\pi_1(\G_*)$-module $\pi_2(\G_*)=\ker(\d)$.

One of the main results of this paper is to show that any crossed module $\d:\G_2\to \G_1$  fits in a commutative diagram
$$\xymatrix{\G_2\ar[d]_{id} \ar[r]^{\dd}& {\bf Z}_1(\G_*)\ar[d]^{{\sf z}_1}\\
	\G_2 \ar[r]^{\d}& \G_1
}
$$
where the top horizontal $\G_2\xto{\dd} {\bf Z}_1(\G_*)$ and right vertical $ {\bf Z}_1(\G_*) \xto {{\sf z}_1} \G_1$ arrows have again crossed module structures. In fact the former is even a braided crossed module, which we call the \emph{centre of the crossed module} $\d:\G_2\to \G_1$ and denote it by ${\mathcal Z}_*(\G_*)$. 
We denote the crossed module corresponding to the right vertical arrow by \(\G_*//{\mathcal Z}_*(\G_*)\). Conceptually, we want to think of the latter as a 2-mathematical analogue of the quotient of the crossed module by its centre, see Subsection \ref{2-maths}.

The centre of  a crossed module is closely related to the  Drinfeld-Joyal-Street centre of a monoidal category \cite{jst}. Namely, it is well-known that any (braided) crossed module defines a (braided) monoidal category \cite{js}. It turns out that the braided monoidal category corresponding to the braided crossed module ${\mathcal Z_*}(\G_*)$ \cite{js} is isomorphic to the centre of the monoidal category corresponding to $\G_*$, see Proposition \ref{9}.

Our definition of ${\mathcal Z}_*(\G_*)$ is based on certain crossed homomorphisms $\G_1\to \G_2$ and has some advantage compared to one based on monoidal categories. Namely, the description of ${\mathcal Z_*}(\G_*)$ in terms of crossed homomorphisms makes it easy to relate the centre of a crossed module to group cohomology. In fact, there are two interesting connections. Firstly, as we will show, $\pi_1({\mathcal Z}_*(\G_*))$ is a subgroup of 
$\H^1(\G_1,\G_*)$, the cohomology of $\G_1$ with coefficients in the crossed module $\G_*$, as defined in  \cite{dg}. Secondly, the essential invariants of ${\mathcal Z}(\G_*)$ are closely related to the low dimensional group cohomology. In fact, one has an isomorphism of groups
$$ \pi_2({\mathcal Z}_*(\G_*))\cong \H^0(\pi_1(\G_*), \pi_2(\G_*))$$
and the group $\pi_1({\mathcal Z}_*(\G_*))$ fits in an exact sequence

\begin{equation}\label{zm} 0\to \H^1(\pi_1, \pi_2)\to \pi_1({\mathcal Z}_*(\G_*))\to  {\sf Z}_{\pi_2(\G_*)}(\pi_1(\G_*)) \xto{g} \H^2(\pi_1, \pi_2),\end{equation}
where
$\pi_1$ and $\pi_2$ denote the groups $\pi_1(\G_*)$ and $\pi_2(\G_*)$, see Lemma \ref{exseq} and part iv) of Proposition \ref{15pr} and \({\sf Z}_{\pi_2(\G_*)}(\pi_1(\G_*))\) is the subgroup of the centre of $\pi_1(\G_*)$ consisting of those elements which act trivially on $\pi_2(\G_*)$.

The main applications of our centre are the following results in homotopy theory (see Proposition \ref{709.pr2} and Theorem \ref{Got=Wh}). Let $\G_*$ be a crossed module  with free $\G_1$. Then there is a homotopy equivalence \[{\mathcal Z}(B\G_*)\simeq B({\mathcal Z}_*(\G_*)).\]
We then give an explicit description of \(G(B\G_*,1)\) as a subgroup of \({\sf Z}_{\pi_2(\G_*)}(\pi_1(\G_*))\), by identifying it with the kernel of the homomorphism \(g\) in the above exact sequence, i.e. we show exactness of the sequence
\[0\to G(X,x_0) \to  {\sf Z}_{\pi_2(\G_*)}(\pi_1(\G_*)) \xto{g} \H^2(\pi_1(\G_*), \pi_2(\G_*)).\]

%We also have an isomorphism
%$$G(B\G_*,1)\cong {\sf Z}_{\pi_2(\G_*)}(\pi_1(\G_*)).$$
%The last group consists of those elements of the centre of $\pi_1(\G_*)$ which act trivially on $\pi_2(\G_*)$.

It should be pointed out that in the 80's Norrie also introduced the notion of a centre of a crossed module \cite{norie}, but our notion differs from hers. We show that there exists a comparison morphism from Norrie's centre to ours, which induces an isomorphism on $\pi_2$ but not on $\pi_1$. The advantage of our definition is the exact sequence (\ref{zm}) which shows that $\pi_1$ of our centre has a nice relation to group cohomology.

For any group $G$, one has the crossed module $\d:G\to Aut(G)$, where $\d(g)$ is the inner automorphism corresponding to $g\in G$. This crossed modules is denoted by ${\sf AUT}(G)$. For $G=D_4$ we compute both Norrie's and our centre. Norrie's centre is ${\sf C_2}\to \{1\}$, while our centre is more complicated and our computations show that $\pi_1({\mathcal Z}_*({\sf AUT}(D_4)))\cong {\sf C_2} \times {\sf C_2}$ and hence they have nonisomorphic  $\pi_1$, see Section \ref{ex45}.

The paper is organised as follows. After the preliminaries in Section \ref{prelim_centre} we introduce and prove the main properties of the centre of a crossed module in Section \ref{4sec}.
The connection to the Gottlieb group is explored in Section \ref{topology}.

As was demonstrated by Loday \cite{jll} there is a generalisation of crossed modules, known as ${\sf cat}^n$-groups, which serve as algebraic models for connected spaces $X$ for which $\pi_i(X)=0$ for all $i>n$. We believe that there exists an extension of our centre to ${\sf cat}^n$-groups and they will have a connection to the centre of such $X$.

\section{Preliminaries on crossed modules and  monoidal categories}\label{prelim_centre}
The material of this section is well known. We included it in order to fix terminology and notations.

\subsection{Crossed modules}\label{13se} Recall that a \emph{crossed module} $\G_*$ is a group homomorphism $\d:\G_{2}\to \G_1$ together with another group homomorphism  $\rho: \G_1\to Aut(\G_{2})$ such that
\begin{align} \d(^xa)&=x\d(a)x\1,  \  \label{CM1}\\
^{\d(b)}a&=bab\1,  \  \label{CM2}
\end{align}
where $a,b\in \G_{2}$ and $x\in {\bf G}_1$. Here and elsewhere we write $^xa$ instead of $\rho(x)(a)$.

We refer to \cite{nat} for an extensive study of crossed modules and their role in homotopy theory. Additionally, we recommend the recent article \cite{Hueb} for  the history of crossed modules and applications.

It follows that ${\sf Im}(\d)$ is a normal subgroup of $\G_1$ and thus 
$$\pi_1(\G_*)=\G_1/\im(\d)$$ is a group.
Moreover, $$\pi_2(\G_*)=\ker(\d)$$ is a central subgroup of $\G_2$ and the action of $\G_1$ on $\G_2$ induces a $\pi_1(\G_*)$-module structure on the abelian group $\pi_2(\G_*)$. 
Thus one has an exact sequence of groups
$$1\to \pi_2(\G_*)\to \G_{2} \xto{\d} \G_{1}\to \pi_1(\G_*) \to 1.$$

A \emph{morphism of crossed modules} $\al_*:\G_*\to \G'_*$ is a pair of group homomorphisms $\al_{2}:\G_{2}\to\G_{2}'$ and $\al_1:\G_{1}\to \G_{1}'$ that preserve the structure of crossed modules, ensuring compatibility with the boundary maps and actions. Crossed modules and their morphisms form a category, denoted by ${\bf Xmod}$.

Recall also that if $\G_*$ is a crossed module, then it gives rise to a monoidal groupoid $\cat(\G_*)$. Objects of $\cat(\G_*)$ are elements of $\G_1$. If $x,y\in \G_1$ then a morphism from $x\to y$ is given by a diagram $x\xto{a} y$, where $a\in \G_2$ satisfies $y=\d(a) x$. The composition of morphisms is defined by
$$\xymatrix{& \d (a) x\ar[dr]^b &\\ x\ar[ur]^a\ar[rr]_{ba}&&\d(b) \d(a) x\,.}$$

One easily checks that in this way one obtains a groupoid, where the identity morphism of an object $x$ is given by $x\xto{1}x$. In fact, $\cat(\G_*)$ has a canonical monoidal category structure. The monoidal structure is denoted by $\cdot$ and the corresponding bifunctor $$\cdot: \cat(\G_*) \times  \cat(\G_*) \to \cat(\G_*)$$ is given on objects by $(x,y)\mapsto x\cdot y=xy$, and on morphisms it is given by
$$(x\xto{a} x')\cdot (y\xto {b} y')= (xy\xto{a\, ^xb} x'y').$$ 
This is well-defined, because
$$\d(a\, ^xb)xy=\d (a)x \d(b)x\1 xy=\d (a) x \d (b)y=x'y'.$$ 

In particular, if $x\xto{a} x'$ is a morphism, and we act by the functor $(-)\cdot y$, we obtain the morphism $xy\xto{a} x'y$, while if we act by the functor $x\cdot (-)$ on a morphism $y\xto{b} y'$, we obtain the morphism $xy \xto{^xb} xy'$. In other words, $a\cdot 1_y=a$ and $1_x \cdot b=\, ^xb$.

We always consider the groupoid $\cat(\G_*)$ as pointed, where the chosen object is $1\in \G_1$. In particular, we will write $\pi_1(\cat(\G_*))$ instead of $\pi_1(\cat(\G_*), 1)$. Then comparing the definitions we see that
$\pi_{i+1}(\G_*)=\pi_i(\cat(\G_*))$ for $i=0,1$.
 
\subsection{ Braided crossed modules}  \begin{De}\label{bcm}
	A \emph{braided crossed module} (BCM for short) $\G_*$  is a  group homomorphism $\d:\G_{2}\to \G_1$ together with a map $\{-,-\}:\G_1 \times \G_1 \to \G_{2}$ such that
	\begin{align} \d\{x,y\}&=[x,y],\  \label{BCM1}\\
	\{\d a,\d b\}&=[a,b], \label{BCM2}\\
	\{\d a,x\}&=\{x,\d a\}\1, \label{BCM3}\\ 
	\{x,yz\}&=\{x,y\}\{x,z\}\{zxz\1x\1, y\}, \label{BCM4}\\ 
	\{xy,z\}&=\, \{x,yzy\1\}\{y,z\}. \label{BCM5} 
	\end{align}
	Here as usual $x,y\in \G_1$ and $a,b\in \G_2$. Recall also that $[x,y]=xyx\1y\1$ is the commutator.
	
	A \emph{symmetric crossed module} (SCM for short) is a BCM for which
		$$\{y ,x\}=\{x,y\}\1 $$
holds for all $x,y\in \G_1$. This condition	is much stronger than (\ref{BCM3}).
\end{De}

The definitions of BCM and SCM, except for the terminology, go back to Conduch\'e \cite[2.12]{conduche}. The relation with braided and symmetric monoidal categories was discovered by Joyal and Street \cite{js}.

Recall also the following definition \cite{4dim}.

 \begin{De} \label{rqm}  A \emph{reduced quadratic module} (RQM for short) is a group homomorphism $\G_2\xto{\d} \G_1$ together with a group homomorphism
$$\G_1^{ab} \otimes \G^{ab}_1 \to \G_2,\ \ \bar{x}\otimes \bar{y}\to \{\bar{x},\bar{y}\},$$
where $K^{ab}=K/[K,K]$ and $\bar{x}$ denotes the image of $x\in K$ in $K^{ab}.$ One requires that $\G_2$ and $\G_1$ are nilpotent groups of class two and for all $a,b\in \G_2$ and $x,y\in \G_1$ the following identities hold:
\begin{align*} 
1&=\{\overline{\d(a)},\bar{x}\}\{\bar{x},\overline{\d(a)}\},\\
[a,b]&=\{\overline{\d(a)},\overline{\d(b)}\},\\
[x,y]&=\d\{\bar{x},\bar{y}\}.
\end{align*} 
\end{De}
It is clear that any RQM is a BCM.

A \emph{morphism of BCM} is a pair of group homomorphisms that preserve the structure of BCM, ensuring compatibility with the boundary maps and the braiding operation.

\begin{Le}\label{2le} Let $\G_*$ be a BCM. Define $\rho:\G_1\to Aut(\G_{2})$ by
	$$^xa:=\{x,\d a\}a.
	$$
	Then one obtains a crossed module.
\end{Le}

\begin{proof} This fact is due to Conduch\'e \cite{conduche}. 
\end{proof} 
\begin{Le} Let $\G_*$ be a BCM. Then $\pi_1(\G_*)$ is an abelian group and the induced action of $\pi_1(\G_*)$ on $\pi_2(\G_*)$ is trivial.
\end{Le}
\begin{proof} By (\ref{BCM1}) we have $[\G_1,\G_1]\in \im(\d)$. Hence $\pi_1(\G_*)=\G_1/\im(\d)$ is abelian.
	For the second part, observe that $a\in \pi_2(\G_*)$ iff $\d(a)=1$. Hence $^xa=\{x,1\}a$. Thus we only need to show that $\{x,1\}=1$. By taking $y=z=1$ in (\ref{BCM5}) one obtains $\{1,1\}=1$. Next, we put $y=z=1$ in  (\ref{BCM4}) to obtain $\{x,1\}=1$.  
\end{proof}
The following important result is due to Joyal and Street \cite{js}.

\begin{Le}\label{4js} Let $\G_*$ be a crossed module. Then there is a one-to-one correspondence between the bracket operations $\{-,-\}:\G_1\times  \G_1\to \G_2$ satisfying the conditions listed in Definition \ref{bcm}  and the braided monoidal category structures on ${\sf Cat}(\G_*)$. Under this equivalence, the element $\{x,y\}$ corresponds to the braid $yx\xto{\{x,y\}} xy$.
\end{Le}

\subsection{The centre of a monoidal category}\label{21s} Let $(\ca, \otimes)$ be a monoidal category. Recall that \cite{jst} the \emph{centre} of  $(\ca, \otimes)$ is the braided monoidal category ${\mathcal Z}(\ca,\otimes)$ which is defined as follows. Objects of ${\mathcal Z}(\ca,\otimes)$ are pairs $(x,\xi)$ where $x$ is an object of $\ca$ and $\xi: (-)\t x\to x \t (-)$ is a natural isomorphism of functors for which the diagram 
$$\xymatrix{ y\t z\t x \ar[dr]_{1_y\t \xi_z }\ar[rr]^{\xi_{y \t z}}&& x\t y \t z
	\\ &y\t x\t z \ar[ur]_{\xi_y\t 1_z}
}$$
commutes for all $y,z\in {\mathcal Ob}(\ca)$. Here and in what follows we write $\xi_y:y\t x\to x \t y$ for the value of $\xi$ on  $y\in {\mathcal Ob}(\ca)$. A morphism $(x,\xi)\to (y,\eta)$ in ${\mathcal Z}(\ca,\otimes)$ is a morphism $f:x\to y$ of $\ca$ such that the diagram
$$\xymatrix{z\t x\ar[d]_{1_z\t f}\ar[r] ^{\xi_z} & x\t z \ar[d]^{f\t 1_z}\\ z\t y \ar[r]_{\eta_z}& y\t z
}$$
commutes for all $z\in{\mathcal Ob}(\ca)$. The monoid structure on ${\mathcal Z}(\ca,\otimes)$ is given by
$$(x,\xi)\t (y\t \eta):=(x\t y, \zeta),$$
where $\zeta_z:z \t x \t y \to x \t y \t z$ is the composite map
$$z \t x \t y \xto{\xi_z\t 1_y} x \t z \t y \xto{1_x \t \eta_ z} x \t y \t z .$$
Finally, the braiding $$c:(x,\xi)\t (y,\eta)\to (y,\eta)\t (x,\xi)$$  
is given by the morphism $\eta_x:x\t y\to y\t x$.

The assignment $(x,\xi)\mapsto x$ obviously extends as a strict monoidal functor ${\mathcal Z}(\ca,\otimes)\to \ca$, which is denoted by ${\sf z}_\ca$.

\section{The centre of a crossed module}\label{4sec}

Let $\G_*$ be a crossed module. In this section we will construct a braided crossed module ${\mathcal Z}_*(\G_*)$ such that the braided monoidal categories $\cat({\mathcal Z}_*(\G_*))$ and ${\mathcal Z}(\cat(\G_*))$ are isomorphic. Because of this fact we will call the crossed module ${\mathcal Z}_*(\G_*)$ the \emph{centre} of $\G_*$. It should be pointed out that ${\mathcal Z}_*(\G_*)$ differs from the centre of $\G_*$ in the sense of \cite{norie}, though there is some relation, which will be discussed in Section \ref{norr}.

The crucial step in the definition of the centre of a crossed module is the group ${\bf Z}_1(\G_*)$. As a set, it is constructed in Definition \ref{z}, and in Lemma \ref{7lem} we equip it with a group structure. The group $\G_1$ acts on ${\bf Z}_1(\G_*)$ (Lemma \ref{10}) and in this way we obtain a crossed module ${\sf z}_1: {\bf Z}_1(\G_*)\to \G_1$ (see Proposition \ref{croq}). The centre  ${\mathcal Z}_*(\G_*)$ has the form $\G_2\xto{\dd} {\bf Z}_1(\G_*)$. The definition of the (braided) crossed module structure on $ {\mathcal Z}_*(\G_*)$ and checking the axioms is given in Lemma \ref{11lema} and Corollary \ref{14p}. The relation to the centre of a monoidal category is given in Proposition \ref{9}. The section \ref{42se} clarifies relations between $\pi_i({\mathcal Z}_*(\G_*) )$ and the group cohomology discussed in the introduction. The fact that ${\bf Z}_1(\G_*)$ is a subgroup of the first nonabelian group cohomology in the sense of \cite{dg} is established in Section \ref{43se}. The map from Norrie's centre to ours is constructed in Section \ref{norr}. In Section \ref{ex45} we compute the centre of the crossed module $G\to Aut(G)$, when $G=D_4$ is the dihedral group of order $8$. 

\subsection{Definition} 
The following is the main definition of this paper.
 
\begin{De}\label{z}Let $\G_*$ be a crossed module. Denote by ${\bf Z}_1(\G_*)$ the set of all pairs $(x, \xi)$ where $x\in \G_1$ and $\xi:\G_1\to \G_2$ is a map satisfying the following identities 
\begin{align}
\d \xi(t)&=[x,t], \ \  \label{ZE1}\\
\xi (\d a )&= \, ^xa\cdot a\1,  \  \label{ZE2}\\
\xi(st)&= \xi(s)\, ^s\xi(t).  \  \label{ZE3}
\end{align}
Here $s,t\in \G_1$ and $a\in \G_2$.
\end{De}

Observe that the last condition says that $\G_1\xto{\xi} \G_2$ is a crossed homomorphism. 

In this section we construct two crossed modules involving ${\bf Z}_1(\G_*)$. The first one is $ {\bf Z}_1(\G_*) \to \G_1$ (see Proposition  \ref{croq} below) and the second one is $\G_2\to {\bf Z}_1(\G_*)$ (see Corollary \ref{14p}). The latter is called the \emph{centre} of $\G_*$ and is in fact a BCM.

We start with a group structure on ${\bf Z}_1(\G_*)$.

\begin{Le} \label{7lem}
\begin{enumerate}[leftmargin=*]
\item [i)] If $(x,\xi), (y,\eta)\in {\bf Z}_1(\G_*)$, then $(xy, \, ^x\eta\cdot \xi)\in {\bf Z}_1(\G_*)$, where $^x\eta\cdot \xi$ is a map $\G_1\to \G_2$ given by
	$$t\mapsto  \, ^x\eta (t)\cdot \xi(t).$$
\item [ii)] The set ${\bf Z}_1(\G_*)$ is a group under the operation
$$(x,\xi)\cdot (y,\eta):= (xy,\,^x\eta\cdot \xi).$$
The unit element is $(1,{\bf 1})$, where $1$ is the identity of $\G_1$ and ${\bf 1}:\G_1\to \G_2$ is the constant map with value $1$.
\end{enumerate}

 \end{Le}

\begin{proof} 
	
\begin{enumerate}[leftmargin=*]
	
\item [i)] Denote by $\zeta$ the function $^x\eta\cdot \xi$. One needs to check three identities: (\ref{ZE1}), (\ref{ZE2}) and (\ref{ZE3}). We have
	\begin{align*} \d \zeta (t)&= \d(^x\eta (t)) 
\d 	 (\xi(t))\\
&= x \d(\eta(t))x\1 \d (\xi(t))\\
&=x[y,t]x\1[x,t]\\
&=xyty\1x\1t\1\\
&=[xy,t]
	\end{align*}
and the condition (\ref{ZE1}) holds.

We also have
\begin{align*} 
\zeta(\d a)&=\, ^x\eta(\d a)\xi(\d a)\\
&= \, ^x(^ya\cdot a\1)(^x a \cdot a\1)\\
&=\, ^{xy}a \cdot ^x(a\1)^x(a) a\1\\
&= \,^{xy}a\cdot a\1. 
\end{align*}
Thus the condition (\ref{ZE2})  also holds.

 Finally, we also have
 \begin{align*} \zeta(st)&=\, ^x(\eta(st))\xi(st) \\
 & =\, ^x\eta(s) ^{xs}\eta(t)\xi(s)^s\xi(t).
  \end{align*}
  On the other hand, we also have
 $$ \zeta(s)^s\zeta(t)=\,^x\eta(s)\xi(s) \, ^{sx}\eta(t)^s\xi(t).
  $$
 To show these expressions are equal, we have to show that $^{xs}\eta(t)\xi(s)=\xi(s) ^{sx}\eta(t)$, or equivalently that
 $\xi(s) ^{sx}\eta(t)\xi(s)\1=\,^{xs}\eta(t)$. To this end, observe that
  \begin{align*} \xi(s) ^{sx}\eta(t)\xi(s)\1&=\,^{\d\xi(s)sx}\eta(t)\\
  &=\,^{[x,s]sx}\eta(t)\\
  &=\, ^{xs}\eta(t)
 \end{align*}
 and the identity (\ref{ZE3}) follows.
 
 The proof of ii) is straightforward and left to the reader.
 
\end{enumerate}
 \end{proof}

It is clear from the definition of the multiplication in ${\bf Z}_1(\G_*)$, that the map ${\sf z}_1:{\bf Z}_1(\G_*)\to \G_1$ given by $${\sf z}_1(x,\xi)=x$$ is a group homomorphism. Our next aim is to show that this is in fact a crossed module.

\begin{Le}\label{10} For any $z\in \G_1$ and any $(x,\xi)\in {\bf Z}_1(\G_*)$ one has
$(^zx, \psi)\in {\bf Z}_1(\G_*)$, where as usual $^zx=zxz\1$ and 
$$\psi(t)=\,^z\xi(^{z\1}t).$$
\end{Le}

\begin{proof}  First of all we need to check that the pair $(^zx,\psi)$ satisfies the identities (\ref{ZE1})-(\ref{ZE3}). In fact, we have
$$\d (\psi(t))=\d (^z\xi(^{z\1}t) )=z[x,^{z\1}t]z\1=
zx\,^{z\1}t x\1 {^{z\1}}(t\1)z\1.$$
On the other hand, the last expression is equal to
$$ zxz\1 tzx\1z\1t\1zz\1= (zxz\1)t(zxz\1)\1t\1,$$ which is the same as $[^zx,t]$ and (\ref{ZE1}) follows.

Next, we have
$$\psi(\d a)=\, ^z\xi(^{z\1}\d(a))=\, ^z\xi(\d(^{z\1}a))=\, ^z(^{x}(^{z\1}a)(^{z\1}a^{-1}))=\, ^{zxz\1}a\cdot a\1
$$
and the equality (\ref{ZE2}) follows.

We also have
$$\psi(st)=\,^z\xi(^{z\1}s\cdot  ^{z\1}t)=^z\left( \xi(^{z\1}s)\cdot ^{z\1 sz}\xi(^{z\1}t\right)=\psi(s)\cdot ^{sz}\xi(^{z\1}t)=\psi(s)\cdot ^s\psi(t),
$$ 
proving the equality (\ref{ZE3}). Hence, $(^zx, \psi)\in {\bf Z}_1(\G_*)$.
\end{proof}

\begin{Pro}\label{croq}
 The construction described in Lemma \ref{10} defines an action of the group $\G_1$ on ${\bf Z}_1(\G_*)$. In this way, the map $${\sf z}_1: {\bf Z}_1(\G_*)\to \G_1$$ is a crossed module. 

\end{Pro} 

\begin{proof} Take two elements $(x_i,\xi_i)$, $i=1,2$ in ${\bf Z}_1(\G_*)$. Assume that
$$(x,\xi)=(x_1,\xi_1)\cdot (x_2,\xi_2)$$ is the product of these elements and $z\in \G_1$. We claim that 
$$^z(x,\xi)=\,^z(x_1,\xi_1)\cdot \, ^z(x_2,\xi_2).$$
Recall that $x=x_1x_2$ and $\xi(t)=\, ^{x_1}\xi_2(t)\cdot \xi_1(t).$ It follows that $$^z(x,\xi)=(^zx_1\, ^zx_2, \psi),$$ where
$$\psi(t)=\, ^z\xi(^{z\1}t)=\, ^{zx_1}\xi_2(^{z\1}t)\cdot ^{z}\xi_1(^{z\1}t).$$
On the other hand, we have 
$$ ^z(x_1,\xi_1) =(^zx_1,\psi_1) \ {\rm and} \  ^z(x_2,\xi_2) =(^zx_2,\psi_2),$$
where $\psi_i(t)=\, ^z\xi_i(^{z\1}t)), i=1,2$.
Hence 
$$^z(x_1,\xi_1)\cdot \, ^z(x_2,\xi_2)=(^zx_1\, ^zx_2, \psi')$$
where
$$\psi'(t)=^{^zx_1z}\xi_2(^{z\1}t)^z\xi_1(^{z\1}t)=^{zx_1}\xi_2(^{z\1}t)^z\xi_1(^{z\1}t)
$$
and the claim follows.

We still need to show the following equalities:
$$^z(^u (x,\xi))=\,^{zu}(x,\xi),$$
 $${\sf z}_0( ^z(x,\xi))=\,^z{\sf z}_0(x,\xi)$$
 and 
 $$^{{\sf z}_0(x,\xi)}(y,\eta)=(x,\xi)(y,\eta)(x,\xi)\1.$$
 
Since the proofs of the first two identities are straightforward, we omit them and check only the validity of the third identity, which is equivalent to
$$^x(y,\eta)(x,\xi)=(x,\xi)(y,\eta).$$
The RHS equals $(xy,t\mapsto\, ^x
\eta(t) \xi(t))$, while the LHS equals
$$(xyx\1,t\mapsto \, ^x\eta(x\1tx))(x,\xi)=(xy,t\mapsto \,^{xyx\1}\xi(t)\cdot ^x\eta(x\1tx)).$$
So we need to check
$$^x\eta(t)\xi(t)=\, ^{xyx\1}\xi(t)\cdot ^x\eta(x\1tx).$$
The last equality is equivalent to
$$\eta(t)\cdot ^{x\1}\xi(t)=\, ^{yx\1}\xi(t)\cdot \eta(x\1tx).$$
We set $t=xs$. Then the above relation is equivalent to
$$\eta(xs)\cdot ^{x\1}\xi(xs)=\, ^{yx\1}\xi(xs)\eta(sx).$$
So we need to check 
\begin{equation}\label{dddd}
\eta(xs)=\, ^{yx\1}\xi(xs)\cdot \eta(sx) \left ( ^{x\1}\xi(xs)\right )\1.
\end{equation}
According to the equation (\ref{ZE2}), we have

\begin{align*}
^{yx\1}\xi(xs)&= \eta (\d( ^{x\1}\xi(xs) ))  \cdot ^{x\1}\xi(xs)\\
&=\eta(x\1\d(\xi(xs))x)  \cdot ^{x\1}\xi(xs)\\
&=\eta(x\1 [x,xs]x)\cdot ^{x\1}\xi(xs)\\
&=\eta([x,s])\cdot ^{x\1}\xi(xs).
\end{align*}
We put $a=\, ^{x\1}\xi(xs)$, $b=\eta(sx)$. Since $$\d(a)=x\1 \d(\xi(xs))x=x\1 [x,xs]x=[x,s],$$ we obtain 
\begin{align*} ^{yx\1}\xi(xs)\cdot \eta(sx) \left ( ^{x\1}\xi(xs)\right )\1&=\eta([x,s])\cdot aba\1 \\
&=\eta([x,s])\cdot\, ^{\d(a)}b\\
&=\eta([x,s])\cdot ^{[x,s]}\eta(sx)\\
&=\eta([x,s]sx)\\
&=\eta(xs).
\end{align*}
Thus the equality (\ref{dddd}) holds and hence the result is proved. 
\end{proof}

\begin{Le}\label{derker} One has an exact sequence
	$$0\to \Der(\pi_1(\G_*),\pi_2(\G_*))\to {\bf Z}_1(\G_*)\xto{{\sf z}_1} \G_1,$$
	where as usual $\Der$ denotes the set of all crossed homomorphisms.
\end{Le}

\begin{proof} The pair $(x,\xi)$ is in the kernel of ${\sf z}_1:{\bf Z}_1(\G_*)\to \G_1$ iff $x=1$. Now Definition \ref{z} shows  that the values of $\xi$ lie in $\pi_2(\G_*)$ and $\xi$ vanishes on the image of $\d:\G_2\to \G_1$. Since it is a crossed homomorphism it factors in a unique way as a crossed homomorphism $\pi_1(\G_*)\to \pi_2(\G_*)$ and hence the result.

\end{proof}

Our next goal is to relate $ {\bf Z}_1(\G_*)$ to the centre of the monoidal category ${\sf Cat}(\G_*)$. We start with the following observation.

\begin{Le}\label{11lema} 
\begin{enumerate}[leftmargin=*]
\item [i)] Let $c\in \G_2$. Then $(\d(c),\zeta_c) \in {\bf Z}_1(\G_*)$, where $\zeta_c:\G_1\to \G_2$ is the map given by $$\zeta_c(t)=c(^tc)\1.$$

\item [ii)]  The map $\dd: \G_2\to {\bf Z}_1(\G_*)$ given by
	$$\dd(c)=(\d (c), \zeta_c)$$
	is a group homomorphism.
	
\item [iii)]  Define the action of ${\bf Z}_1(\G_*)$ on $\G_2$ by $$^{(x,\xi)}a:=\, ^xa$$
Then $\dd: \G_2\to {\bf Z}_1(\G_*)$ is a crossed module.
\end{enumerate}
	
\end{Le}
\begin{proof}
\begin{enumerate}[leftmargin=*]
	\item [i)] We have to check that the pair $(\d(c),\zeta_c)$ satisfies the identities (\ref{ZE1})-(\ref{ZE3}). In fact, we have
	$$
	\d (\zeta_c(t))=\d (c(^tc)\1)= \d c (t\d c t\1)\1=\d c t \d c \1 t\1=[\d c,t].
	$$
	Thus (\ref{ZE1}) holds. Next, we also have
	\begin{align*} \zeta_c(\d a)&= c(^{\d a}c)\1\\
	&= c(aca\1)\1\\
	&=[c,a]\\
	&=(^{\d c}a) a\1.
	\end{align*}
	Hence (\ref{ZE2}) holds. Finally, we have
	\begin{align*} \zeta_c(s)\, (^s\zeta_c(t))&= c(^sc)\1\cdot  ^s(c(^tc)\1)\\
	&=c(^sc)\1(^sc) (^{st}c)\1\\
	&=c(^{st}c)\1\\
	&=\zeta_c(st).
	\end{align*}
	This finishes the proof.
	
\item [ii)]  Clearly $(\d b,\zeta_b)(\d c,\zeta_c)=(\d (bc), ^{\d b}\zeta_c\cdot \zeta_b)$, where
	\begin{align*} (^{\d b}\zeta_c\cdot \zeta_b)(t)&=\, b(c \, ^t(c\1))b\1(b\, ^t (b\1))\\
	&= \, bc \, ^t(c\1)\, ^t (b\1)\\
	&= \, bc \, ^t(c\1b\1)\\
	&=\zeta_{bc}(t).
	\end{align*}
	Thus $(\d b,\zeta_b)(\d c,\zeta_c)=(\d (bc),\zeta_{bc})$. Hence $\dd$ is a group homomorphism.

\item [iii)]  We have to check the identities (\ref{CM1}) and (\ref{CM2}). For the second, observe that
$$^{\dd(b)}a=\, ^{(\d(b), \theta_b)}a=\, ^{\d(b)}a=bab\1.$$
For the first one, we need to check that
$$\dd(^{(x,\xi)}a)=(x,\xi)\dd(a) (x,\xi)\1.$$ 
Equivalently,
$$\dd(^xa)(x,\xi)=(x,\xi)\dd(a),$$
which is the same as
$$(\d (^xa),\zeta_{^xa})(x,\xi)=(x,\xi)(\d(a), \zeta_a).$$
The first coordinate of the LHS is equal to $x\d(a)x\1\cdot x=x\d(a)$, the same as the first coordinate of the RHS. Hence we need to show that for all $t\in \G_1$ one has
$$\, ^{x\d(a)x\1}\xi(t)\cdot ^xa(^{tx}a)\1=\, ^x(a(^ta)\1)\xi(t).$$
Equivalently,
$$^{\d(a)\cdot x\1}\xi(t)\cdot a \cdot ^{x\1tx}(a\1)=a\cdot ^t(a\1)^{x\1}\xi(t).$$
Using the equality $^{\d(a)}b=aba\1$ and cancelling out $a$, we see that the equality can be rewritten as
$$^{x\1}\xi(t)\cdot ^{x\1tx}(a\1)=\, ^t(a\1)\cdot ^{x\1}\xi(t).$$
Now this is obviously equivalent to
$$\xi(t)\cdot ^{tx}(a\1)=\, ^{xt}(a\1)\xi(t)$$
and hence to
$$\xi(t)\cdot ^{tx}(a\1)\xi(t)\1=\, ^{xt}(a\1).$$
Based on (\ref{ZE2}), we have $^xb=\xi(\d(b))b$. Hence we can write
\begin{align*}
^{xt}(a\1)&=\xi(\d (^t(a\1)))\cdot ^t(a\1)\\
&=\xi(t\d(a)t\1)\cdot ^t(a\1)\\
&=\xi(t)\cdot ^t\xi(\d(a\1)\cdot t\1)\cdot ^t(a\1)\\
&=\xi(t)\cdot ^t\xi(\d(a\1))\cdot ^{t\d(a\1)}\xi(t\1)\cdot ^t(a\1)\\
&=\xi(t)\cdot ^{tx}(a\1)\cdot ^ta \cdot ^t(a\1\xi(t\1)a)\cdot ^t(a\1)\\
&=\xi(t)\cdot ^{tx}(a\1)\xi(t)\1
\end{align*}
and we are done.
\end{enumerate}
\end{proof}

 \begin{Pro}\label{9} 
 	
 \begin{enumerate}[leftmargin=*]
 	
 \item [i)] Let $(x,\xi)\in {\bf Z}_1(\G_*)$. Then for any $y\in \G_1$ we have an arrow $$yx\xto{\xi(y)} xy$$ in the category $\cat(\G_*).$
 	
 \item [ii)] By varying $y$ one obtains a natural isomorphism of functors \[(-)\cdot x\xto{\bar{\xi}} x\cdot (-).\]
 	
  \item [iii)]  The pair $(x, \bar{\xi})$ is an object of ${\mathcal Z}(\cat(\G_*))$.
 	
 \item [iv)] The map $${\bf Z}_1(\G_*)\to {\mathcal Ob}({\mathcal Z}(\cat(\G_*))), \ \ (x,\xi)\mapsto (x,\bar{\xi})$$   is a bijection. 
 	
 \item [v)]  If $(x,\xi),(y,\eta)\in {\bf Z}_1(\G_*)$ and $a\in \G_2$ is an element such that $y=\d(a)x$, then the arrow $x\xto{a} y$ of ${\sf Cat}(\G_*)$
 	defines a morphism in \({\mathcal Z}(\cat(\G_*))\) iff $$(y,\eta(t))=\dd(a)(x,\xi).$$
 	Hence the map ${\bf Z}_1(\G_*)\to {\mathcal Ob}({\mathcal Z}(\cat(\G_*)))$ constructed in iv) extends to an isomorphism of monoidal categories 
 	$${\sf Cat}({\mathcal Z}_*(\G_*))\to {\mathcal Z}(\cat(\G_*)).$$
\end{enumerate}
 \end{Pro}
 \begin{proof} 
 	\begin{enumerate}[leftmargin=*]
 		
 \item [i)] We have to check that $\d (\xi(y)) yx=xy$, which is a consequence of the equality (\ref{ZE1}).
 	
 \item [ii)] Take a morphism $y\xto{a} z$ in $\cat(\G_*)$. Consider the diagram
 	$$\xymatrix{yx \ar[d]_{a\cdot 1_x } \ar[r]^{\xi(y)}& xy \ar[d]^{1_x\cdot a}\\
 		zx\ar[r]_{\xi(z)} &xz	
 	}$$

 Since $a\cdot 1_x=a$, $1_x\cdot a=\,^xa$ and $z=\d (a)y$, we have
 \begin{align*} \xi(z)\circ (a\cdot 1_x)&=\xi(\d (a) y) 
 a\\ &= \xi(\d (a)) \, ^{\d a}\xi(y)a \\ & =(^xa\cdot a\1)(a\xi(y)a\1)a\\ & = \, ^xa \xi(y)\\ & = (1_x \cdot a)\circ \xi(y).
 \end{align*}
 Thus the above diagram commutes and the result follows. 
 
  \item [iii)] We have to check the commutativity of the triangle in Section \ref{21s}. It requires us to verify that
 $$(\xi(y)\cdot 1_z ) \circ (1_y \cdot \xi(z))=\xi(yz)$$
 which is an immediate consequence of the equality (\ref{ZE3}), because $\xi(y)\cdot 1_z=\xi(y)$ and $1_y \cdot \xi(z)=\,^y \xi(z)$.
 
  \item [iv)] Take any object $(x,\bar{\xi})$ of the category ${\mathcal Z}(\cat(\G_*))$. By definition $x$ is an object of $\cat(\G_*)$, thus $x\in \G_1$. Moreover $\bar{\xi}$ is a natural isomorphism $(-)\cdot x\to x\cdot (-)$. The value of $\bar{\xi}$ on an object $y$ is a morphism  $yx\to xy$  of $\cat(\G_*)$ denoted by ${\xi_y}$. Thus $\xi_y\in \G_2$ satisfies the condition $\d \xi_y yx =xy$. It follows that $y\mapsto \xi_y$ defines a map $\xi:\G_1\to \G_2$ satisfying the condition (\ref{ZE1}). By definition of the centre, we have
 $$\xi_{yz}=\xi_y\cdot ^y\xi_z$$
 and the condition (\ref{ZE3}) follows. This implies that $\xi_1=1$. Finally, the condition (\ref{ZE2}) follows from the naturality of $\xi$. In fact, the commutative square in the proof of part ii) implies that for any $y\xto{a}z$ we have 
 $$\xi(z) a=\, ^xa \xi(y).$$
 We can take $y=1$. Then $z=\d a$ and hence
 $$\xi(\d a) a=\,^xa$$
 proving (\ref{ZE2}). Thus we have constructed the inverse map 
 $${\mathcal Ob}({\mathcal Z}(\cat(\G_*)))\to {\bf Z}_1(\G_*)$$
 proving the statement.
 
\item [v)] The diagram
 \[\xymatrix{z\t x\ar[d]_{1_z\t a}\ar[r] ^{\xi(z)} & x\t z \ar[d]^{a\t 1_z}\\ z\t y \ar[r]_{\eta(z)}& y\t z
 }\]
commutes  iff $\eta(z)=a\xi(z)(^za)\1$ and this happens for all $z$ iff
$$(y,\eta)=\dd(a)(x,\xi).$$
In fact, we have 
\begin{align*}
\dd(a)(x,\xi)&=(\d(a)x, z\mapsto\, ^{\d(a)}\xi(z)\zeta_a(z))\\ &=(\d(a)x, z\mapsto a\xi(z)(^za)\1).
\end{align*}
This fact already implies that the functor 
${\sf Cat}({\mathcal Z}_*(\G_*))\to {\mathcal Z}(\cat(\G_*))$
is an isomorphism of categories. Hence, we only need to check that the monoidal structures in both categories are compatible under this isomorphism.
Assume that $(x,\xi), (y,\eta)\in {\bf Z}_1(\G_*)$. Denote by $(x,\bar{\xi})$ and $(y,\bar{\eta})$ the corresponding objects of ${\mathcal Z}(\cat(\G_*))$. According to the definition of the monoidal structure on ${\mathcal Z}(\cat(\G_*))$, we have
\[(x,\bar{\xi}) \cdot (y,\bar{\eta}) =(xy,\bar{\zeta}),\]
 where $\bar{\zeta}_z$ is the composite
 \[zxy\xto{\xi(z)\cdot 1_y}xzy\xto{1_x\cdot \eta(y)} xyz\]
 which is the same as 
 $^x\eta(y) \xi(z)$ and we are done.
 \end{enumerate}
 \end{proof}
 
\begin{Co}\label{14p} The map $\G_2\xto{\dd} {\bf Z}_1(\G_*)$ together with the bracket
$$\{(x,\xi),(y,\eta)\}:=\xi(y)$$
defines a braided crossed module structure on $${\mathcal Z}_*(\G_*):=(\G_2\xto{\dd} {\bf Z_1}(\G_*) ).$$
\end{Co}
\begin{proof} Since the monoidal structure on $ {\mathcal Z}(\cat(\G_*))$ is braided, the same will be true for ${\sf Cat}({\mathcal Z}_*(\G_*))$, if we transport the structure under the constructed isomorphism. Hence the result follows from Lemma \ref{4js}.

\end{proof}

\subsection{2-categorical meaning of the main construction}\label{2-maths}
To summarise, we see that the crossed module $\d:\G_2\to \G_1$  fits in a commutative diagram
$$\xymatrix{\G_2\ar[d]_{id} \ar[r]^{\dd}& {\bf Z}_1(\G_*)\ar[d]^{{\sf z}_1}\\
	\G_2 \ar[r]^{\d}& \G_1,
}
$$
where both $\G_2\xto{\dd} {\bf Z}_1(\G_*)$ and $ {\bf Z}_1(\G_*) \xto {{\sf z}_1} \G_1$ are crossed modules, the first one is even a braided crossed module.
The vertical arrows form a morphism of crossed modules
$${\sf z}_*:{\mathcal Z}_*(\G_*)\to \G_*.$$
The crossed module $ {\bf Z}_1(\G_*) \xto {{\sf z}_1} \G_1$ is denoted by $\G_*//{\mathcal Z}_*(\G_*)$. It is an easy exercise to see that one has an exact sequence
$$0\to \pi_1({\mathcal Z}_*(\G_*))\to \pi_1(\G_*)\to \pi_1(\G_*//{\mathcal Z}_*(\G_*))\to $$ 
$$\to \pi_0({\mathcal Z}_*(\G_*))\to \pi_0(\G_*)\to \pi_0(\G_*//{\mathcal Z}_*(\G_*))\to 0.$$

This suggests that we can think of \({\mathcal Z}_*(\G_*)\), \(\G_*\) and \(\G_*//{\mathcal Z}_*(\G_*)\) as fitting into a 2-mathematical analogue of a short exact sequence

\[0\to {\mathcal Z}_*(\G_*)\to \G_* \to \G_*//{\mathcal Z}_*(\G_*) \to 0.\]

However, in this paper, we will not give a formal definition of short exact sequences in the 2-mathematical sense.

\subsection{On homotopy groups of ${\mathcal Z}_*(\G_*)$}\label{42se} Let $\G_*$ be a crossed module. We have constructed a BCM ${\mathcal Z}_*(\G_*)$. In this section we investigate the homotopy groups $\pi_i({\mathcal Z}_*(\G_*))$, $i=1,2$ of this crossed module. 
The case $i=2$ is easy and the answer is given by the following lemma.
\begin{Le} \label{exseq} Let $\G_*$ be a crossed module. Then
	$$\pi_2({\mathcal Z}_*(\G_*) )\cong \H^0(\pi_1(\G_*),\pi_2(\G_*)).$$
\end{Le}
\begin{proof} By definition $a\in \pi_1({\mathcal Z}_*(\G_*) )$ iff $\dd(a)=(1,{\bf 1})$, thus when $\d (a)=1$ and $\zeta_a(t)=1$ for all $t\in \G_1$. These conditions are equivalent to the conditions $a\in \pi_2(\G_*)$ and $^ta=a$ for all $t\in \G_1$ and hence the result.
\end{proof}

On the other hand, to obtain information on $\pi_1({\mathcal Z}_*(\G_*) )$, we need to fix some notation. The main result is formulated in terms of an exact sequence, see  Proposition \ref{15'pr}. 

For a group $G$, we let ${\mathcal Z}(G)$ denote the centre of $G$. Moreover, if $H$ is a $G$-group, we set
$$st_H(G)=\{g\in G|\, ^gh=h \ {\rm for \ all}\ h\in H\}.$$
It is obviously a subgroup of $G$. The intersection of these two subgroups is denoted by ${\sf Z}_{H}(G)$. Thus $g\in {\sf Z}_{H}(G)$ iff $^gx=x$ and $^gh=h$ for all $x\in G$ and $h\in H$.

Let $\G_*$ be a crossed module. In this case we have defined two groups
$${\sf Z}_{\pi_2(\G_*)}(\pi_1(\G_*)) \ \ {\rm and}
\ \  \ {\sf Z}_{\G_2}(\G_1).$$
We will come back to the second group in Section \ref{norr}. Now we relate the first group to $\pi_1(\mathcal Z_*(\G_*))$. To this end we fix some notation. If $x\in \G_1$, then we let ${\sf cl}(x)$ denote the class of $x$ in $\pi_1(\G_*)$. 
Take now an element $(x,\xi)\in {\bf Z}_1(\G_*)$. Accordingly, ${\sf cl}(x,\xi)$ denotes the class of $(x,\xi)$ in $\pi_1({\mathcal Z}_*(\G_*))$. 
I claim that the class ${\sf cl}(x)\in \pi_1(\G_*)$ belongs to $${\sf Z}_{\pi_2(\G_*)}(\pi_1(\G_*)).$$ In fact, the equation \ref{ZE1} of Definition \ref{z} implies that $[x,y]=\d \xi(y)$ and hence ${\sf cl}([x,y])=0$ in  $\pi_0(\G_*)$. It follows that ${\sf cl}(x)$ is central in $\pi_1(\G_*)$. Moreover, the equation \ref{ZE2} of Definition \ref{z} tells us that $\xi(\d a)=a(^xa)^{-1}$. In particular, if $a\in\pi_2(\G_*)$ (i.e. $\d (a)=1$) then $^xa=a$ and the claim follows. Thus we have defined the group homomorphism
$${\bf Z}_1(\G_*)\xto{\omega'} {\sf Z}_{\pi_2(\G_*)}(\pi_1(\G_*))$$
by $ \omega'(x,\xi):={\sf cl}(x).$

\begin{Pro}\label{15pr} 
	\begin{enumerate}[leftmargin=*]
		
		\item [i)] For any crossed module $\G_*$, the composite map
		\[\G_2\xto{\dd} {\bf Z}_1(\G_*)\xto{\omega'} {\sf Z}_{\pi_2(\G_*)}(\pi_1(\G_*))\]
		is trivial. Hence the map $\omega'$ induces a group homomorphism
		\[\omega: \pi_1({\bf Z}_*(\G_*))\to  {\sf Z}_{\pi_2(\G_*)}(\pi_1(\G_*)).\]
		
		\item [ii)] For a $1$-cocycle $\phi: \pi_1(\G_*)\to \pi_2(\G_*)$, the pair $(1,\tilde{\phi})\in {\bf Z}_1(\G_*)$, where $\tilde{\phi}:\G_1\to \G_2$ is the composite map
		\[\G_1\to \pi_1(\G_*)\xto{\phi} \pi_2(\G_*)\hookrightarrow \G_2.\]
		Moreover, the assignment \(\phi\mapsto {\sf cl}(1,\tilde{\phi})\) induces a group homomorphism
		\[f: H^1(\pi_1(\G_*), \pi_2(\G_*))\to \pi_1({\mathcal Z}_*(\G_*)).\]
		
		\item [iii)] These maps fit in an exact sequence
		\[0\to \H^1(\pi_1(\G_*), \pi_2(\G_*))\xto{f} \pi_1({\mathcal Z}_*(\G_*))\xto{\omega}  {\sf Z}_{\pi_2(\G_*)}(\pi_1(\G_*)).\]
	\end{enumerate}
\end{Pro}

\begin{proof}
	
	\begin{enumerate}[leftmargin=*]
		
		\item [i)] Take $a\in \G_2$. By construction $\dd(a)=(\d(a),\zeta_a)$. Hence \[\omega'\dd (a) ={\sf cl}(\d(a))=1.\]
		
		\item [ii)] Since $\phi$ is a 1-cocycle, $\tilde{\phi}$ satisfies the condition \ref{ZE3} of Definition \ref{z}. Next, the values of $\tilde{\phi}$ belong to $\pi_2(\G_*)$, so $\d \tilde{\phi}=1$ and the condition (\ref{ZE1}) follows. Finally, $\tilde{\phi}(\d a)=\phi([\d a])=\phi[1]=1$ and (\ref{ZE2}) also holds. It remains to show that if $\phi(t)=\, b(^tb)\1$, for an element $b\in\pi_2(\G_*)$, then ${\sf cl}(1,\tilde{\phi})=1$, but this follows from the fact that $\dd(b)=(1,\tilde{\phi}).$
		
		\item [iii)] Exactness at $\H^1(\pi_1(\G_*), \pi_2(\G_*))$: Assume $\phi:\pi_1(\G_*)\to \pi_2(\G_*)$ is a 1-cocycle such that ${\sf cl}(1,\tilde{\phi})$ is the trivial element in $\pi_1({\mathcal Z}_*(\G_*))$. That is, there exists a $c\in \G_2$ such that $\dd(c)=(1,\tilde{\phi})$. Thus $\d (c)=1$ and $\tilde{\phi}(t)=c(^tc)\1$. So $c\in \pi_2(\G_*)$ and the second equality implies that the class of $\phi$ is zero in $\H^1(\pi_1(\G_*), \pi_2(\G_*))$, proving that $f$ is a monomorphism.
		
		Exactness at $\pi_1({\mathcal Z}_*(\G_*))$: First take a cocycle $\phi:\pi_1(\G_*)\to \pi_2(\G_*)$. Then
		$$\omega\circ f([\phi]) =\omega({\sf cl}(1,\tilde{\phi}))
		={\sf cl}(1)=1.$$
		Take now an element $(x,\xi)\in {\mathcal Z}_1(\G_*)$ such that ${\sf cl}(x,\xi)\in \ker(\omega)$. Thus $x=\d a$ for \(a\in \G_2\). Then $[(x,\xi)]=[(y,\eta)]$, where $(y,\eta)=(x,\xi)\dd(a)\1$. Since $y=x\d a\1=1$, we see that $\d \eta(t)=[1,t]=1$ and $\eta(\d c)=1$. So $\eta= \tilde{\phi}$, where $\phi:\pi_1(\G_*)\to \pi_2(\G_*)$ is a 1-cocycle. Thus $f([\phi])=[(1,\eta)]=[(x,\xi)]$ and exactness at $\pi_1({\mathcal Z}_*(\G_*))$ follows.
	\end{enumerate}
\end{proof}

Our next aim is to define the homomorphism $$g: {\sf Z}_{\pi_2(\G_*)}(\pi_1(\G_*))\to \H^2(\pi_1(\G_*), \pi_2(\G_*))$$ and extend the exact sequence constructed in part iii) of Proposition \ref{15pr}. This will be based on the following result.

\begin{Pro}\label{15''pr}
	Take an element $x\in \G_1$ such that ${\sf cl} (x)$ is in the centre of $\pi_1(\G_*)$. Then there exists a map $\psi:\G_1\to\G_2$ such that $\d \psi(t)=[x,t]$  
	for all $t\in \G_1$. Consider the expression 
	\[\bar{\theta}(s,t):=\psi(s)\, ^s\psi(t)\psi(st)\1.\]
	Then $\bar{\theta}(s,t)\in \pi_2(\G_*)$ and the map $\bar{\theta}$ satisfies the 2-cocycle condition: $$^s\bar{\theta}(t,r) \bar{\theta}(s,tr)
	=\bar{\theta}(s,t)\bar{\theta} (st,r).$$
	Moreover, the class of $\bar{\theta}$ in $\H^2(\G_1,\pi_2(\G_*))$ is independent of the choice of $\psi$.
\end{Pro} 

\begin{proof} By assumption, we have $[x,t]\in \im(\d)$ for all $t\in\ G_1$. Thus we can choose $\psi(t)\in \G_2$ such that $[x,t]=\d(\psi(t))$. 
	By definition we have the equality
	$$\bar{\theta}(s,t)\psi(st)=\psi(s)\, ^s\psi(t).$$
	Applying $\d$ to this equation, we obtain
	$$\d (\bar{\theta}(s,t))[x,st]=[x,s] \, ^s[x,t].$$
	Since $[x,yz]=[x,y] \, ^y[x,z]$ holds in any group, it follows that $\d (\bar{\theta}(s,t))=1$. Thus $\bar{\theta}(s,t)\in \pi_2(\G_*)$.
	In particular, $\bar{\theta}(s,t)$ is a central element of $\G_2$.
	
	To show that the $2$-cocycle condition holds, we can write
	\begin{align*}
		\bar{\theta}(s,t)\bar{\theta} (st,r)\psi(str)&=\bar{\theta}(s,t)\bar{\psi}(st)\, ^{st}\bar{\psi}(r)\\
		&=\bar{\psi}(s)^s\bar{\psi}(t)^{st}\bar{\psi}(r)\\
		&=\psi(s)\, ^s\left (\bar{\theta}(t,r)\psi(tr)\right ).
	\end{align*}
	Since $^s\bar{\theta}$ is central, we obtain
	$$\bar{\theta}(s,t)\bar{\theta} (st,r)\psi(str)=\, 
	^s\bar{\theta(t,r)}\bar{\theta}(s,tr)\psi(str).$$
	After cancelling out $\psi(str)$, we obtain the expected property.
	
	If $\psi_1$ also satisfies the condition $\d \psi_1(s)=[x,t]$, we see that $\psi_1(s)=\psi(s) \phi(s)$, where $\d \phi(s)=1.$ Hence $\phi$ takes values in $\pi_2(\G_*)$. Since the last subgroup is closed under the action of $\G_1$ we also have $^s\phi(t)\in \pi_2(\G_*)$. In particular both $\phi(s)$ and $^s\phi(t)$ are central. It follows that
	\[\bar{\theta}_1(s,t)=\bar{\theta}(s,t) \phi (s)\, ^s\phi(t)\phi(st)\1\]
	and hence the result.  
\end{proof}

Our next aim is to show that if one chooses the map $\psi$ from Proposition \ref{15''pr} more carefully, then one can achieve that $\bar{\theta}$ factors through $\pi_1(\G_*)\times \pi_1(\G_*)$ and hence one obtains an element in $\H^2(\pi_1(\G_*), \pi_2(\G_*))$. Recall that ${\sf cl}$ denotes the canonical surjection $$\G_1\xto{\sf cl}\pi_1(\G_*)=\frac{\G_1}{\im(\d)}.$$

\begin{Pro}\label{15'pr}
	Take an element $x\in \G_1$ such that ${\sf cl}(x)\in {\sf Z}_{\pi_2(\G_*)}(\pi_1(\G_*))$ 
	and choose a section $\al:\pi_1(\G_*)\to \G_1$ of ${\sf cl}$ such that  $\al(1)=1$. Then there is a map $$\beta:\pi_1(\G_*)\to \G_2$$ such that
	\begin{equation}\label{ziaagi}\d(\beta(z) )=[x,\al(z)]\end{equation}
	for all $z\in \pi_1(\G_*)$. Moreover, the formula 
	\begin{equation}\label{psisf}
		\psi(s)= \,^xa\beta({\sf cl}( s)) a^{-1}
	\end{equation}
	gives rise to a well-defined map  $\psi:\G_1\to\G_2$. Here $a$ is an element in $\G_2$ for which
	\begin{equation}\label{tolobasa}
		s=\d(a)\al({\sf cl}( s)).
	\end{equation}
	Then $\psi$ and the corresponding $\bar{\theta}$ (see Proposition \ref{15''pr}) satisfy the following properties:
	\begin{enumerate}[leftmargin=*]
		\item [i)] $\d \psi(t)=[x,t]$.
		
		\item [ii)] $\psi(\d a)=\, ^xa \cdot a\1$.
		
		\item [iii)]  $\bar{\theta}(\d (a),\al(z))=1$.
		
		\item [iv)] $\bar{\theta}(\al(z),\d (a))=1$.
		
		\item [v)] $\bar{\theta}(\d (a), \d (b))=1$.
		
		\item [vi)] The function $\bar{\theta}$ factors through $\pi_1(\G_*)\times \pi_1(\G_*)\to \pi_2(\G_*)$, and hence defines a class $\theta$ in $\H^2(\pi_1(\G_*),\pi_2(\G_*))$.
		
		\item [vii)] The class $\theta$ is independent of the choice of $\psi$ and $\alpha$.
		
		\item [viii)] The assignment $x\mapsto \theta$ defines a group homomorphism

		\[g: {\sf Z}_{\pi_2(\G_*)}(\pi_1(\G_*))\to \H^2(\pi_1(\G_*), \pi_2(\G_*)),\]
		which fits in the exact sequence
		\[0\to \H^1(\pi_1(\G_*), \pi_2(\G_*))\xto{f} \pi_1({\mathcal Z}_*(\G_*))\xto{\omega}  {\sf Z}_{\pi_2(\G_*)}(\pi_1(\G_*)) \xto{g} \H^2(\pi_1(\G_*), \pi_2(\G_*)).\]
	\end{enumerate}		
	
\end{Pro}
\begin{proof} By the condition on $x$ one has $[x,y]\in \im(\d)$ for all $y\in \G_1$. Thus $[x,y]=\d (m_y)$ for some $m_y\in \G_2$. For a given $z\in \pi_1(\G_*)$ we take $y=\al(z)$ and denote the corresponding  $m_{y}$ by $\beta(z)$. Then we have $[x,\al(z)]=\d(\beta(z))$.
	
	Clearly any element $s\in \G_1$ can be written as in the formula (\ref{tolobasa}) for some $a\in\G_2$. Let us show that the expression $ \,^xa\psi(\al({\sf cl}( s)) a^{-1}$ in the equality (\ref{psisf}) does not depend on the choice of $a$ in the equality (\ref{tolobasa}). In fact, if we have another decomposition  $s=\d(b)\al({\sf cl}(s))$, then $\d(b)=\d(a)$. It follows that $b=ac$, where $\d(c)=1$. Thus  $c\in\pi_2(\G_*)$ and in particular it is central and since ${\sf cl}(x)\in {\sf Z}_{\pi_2(\G_*)}(\pi_1(\G_*))$ we also have $^xc=c$. Thus we have
	\[^xb\psi(\al({\sf cl}( s))b^{-1} =\,^xa\,^xc\cdot\psi(\al({\sf cl}( s)) c^{-1}a^{-1}= \,^xa \psi(\al({\sf cl}( s)) a^{-1}.\]
	Hence the function $\psi$ is well-defined. 
	
	Now we verify the properties i)-viii).  
	\begin{enumerate}[leftmargin=*]
		\item [i)]
		Observe that
		\begin{align*}
			\d (\psi(s))&=\d(^xa\psi(\al({\sf cl}(s)) a^{-1} )\\ &=x\d(a)x^{-1}[x,\al({\sf cl}( s))]\d(a)^{-1}\\&=x\d(a)\al({\sf cl}(s))x^{-1}\al({\sf cl}( s))^{-1}\d (a)^{-1}\\ &=[x, \d(a) \al({\sf cl}( s))]
		\end{align*}
		Since $s= \d (a)\al({\sf cl}( s))$, we see that \(\psi\) satisfies the condition i).
		
		\item [ii)] If $s=\d (a)$, we have ${\sf cl}(s)=1$ and thus $\psi(\d(a))=\, ^x a\cdot a^{-1}$ and hence the condition ii) also holds.
		
		\item [iii)] By taking $a=1$  and $s=\al(z)$ in (\ref{psisf}), we obtain $\psi(\al(z ))=\beta(z).$
		By property ii) we have $\psi(\d (a))= \, ^xa\cdot a^{-1}$. Therefore
		$$\psi( \d (a))\, ^{\d (a)}\psi(\al(z))=\, ^xa\cdot a^{-1} a \psi(\al(z))a^{-1}=\, ^xa\cdot \beta(z)a^{-1}=\psi(\d (a)\al(z))$$
		and the condition iii) follows.
		
		\item [iv)] We need to show that \begin{equation}\label{dtheta} \psi(\al(z)\d (a))=\psi(\al(z))^{\al(z)}\psi(\d (a)).
		\end{equation}
		The RHS is equal to $\beta(z)\, ^{\al(z)x}a\,( ^{\al(z)}{a^{-1}})$. Since $\al(z)\d (a)=\d (b)\al(z)$, where $b=\, ^{\al(z)}a$, we can rewrite the LHS:
		\begin{align*} \psi(\al(z)\d (a))= &\psi(\d (b) \al(z))\\
			=&\, ^xb\beta(z) b^{-1}\\
			=&\,^{x\alpha(z)}a \beta(z)\, (^{\al(z)}a^{-1}).
		\end{align*}
		Comparing these expressions we see that the equality (\ref{dtheta}) is equivalent to
		\begin{equation}\label{d1theta}^{x\alpha(z)} a\beta(z)=\beta(z)\, ^{\alpha(z)x}a.
		\end{equation} 
		We have 
		\begin{align*} 
			\beta(z)\, ^{\alpha(z)x}a=& [\beta(z),\, ^{\alpha(z)x}a] \, ^{\alpha(z)x}a\beta(z)\\
			=&\beta(z)\, ^{\alpha(z)x}a\beta(z)^{-1}\beta(z)\\
			=& \, ^{\d (\beta(z)) \alpha(z)x}a\beta(z)\\
			=&\,^{x\alpha(z)} a\beta(z)
		\end{align*}
		and part iv) is proved. Here we used the fact that $\d \beta(z)=[x,\alpha(z)]$.
		
		\item [v)] We have $\psi(\d(a) \d(b))=\psi (\d(ab))=\, ^x (ab) \, (ab)^{-1}$. On the other hand $$ \psi(\d(a))\, ^{\d(a)} \psi(\d (b))=\, ^x a\, a^{-1}  a ^xb \ b^{-1}a^{-1}=\, ^x(ab) \, (ab)^{-1}$$ and the result follows.
		
		\item [vi)] In the 2-cocycle condition from Proposition \ref{15''pr} we first put $s=\d(a)$, $t=\d (b)$ $r=\al(z)$ to obtain 
		\[^{\d(a)}\bar{\theta}(\d (b),\al(z))\bar{\theta}(\d (a),\d (b)\al(z)) =\bar{\theta}(\d(a),\d(b))\bar{\theta}(\d(ab), \al(z))\]
		Use the relations iii)-v) to obtain $\bar{\theta}(\d (a),\d (b)\al(z))=1$. Thus $\bar{\theta}(\d (a),s)=1$ for all $s\in \G_1$. Quite similarly
		$\bar{\theta}(s,\d (a))=1$ for all $s$. Now we put $r=\d (a)$ to obtain $\bar\theta(s,t\d (a))
		=\bar\theta (s,t)$, showing that the map $\bar\theta(s,-)$ factors through the group $\pi_1(\G_*)$. Similarly for the first argument.
		
		\item [vii)] Now we prove that the class $\theta$ is independent of the choices which we made, namely of $\al$, $\beta$ and $a$ in equality (\ref{tolobasa}). We already proved that for chosen $\al$ and $\beta$ the function $\psi$ (and hence $\bar{\theta}$) is independent of the choice of $a$. Assume $\al$ is chosen and we have $\beta_1$ and $\beta$ for which the equality (\ref{ziaagi}) holds. Then there is a map $\gamma:\pi_1(\G_*)\to \G_2$ such that $\beta_1(z)= \gamma(z)\beta(z)$. Since $\d(\gamma(z))=1$, we see that $\gamma(z)$ and $^s\gamma(z)$ are central. It follows that
		\[\psi_1(s)=\,^xa\beta_1({\sf cl}(s))a^{-1}=\psi(s)\gamma({\sf cl}(s)).\]
		It follows that
		\[\bar{\theta}_1(s,t)=\bar{\theta}(s,t) \gamma ({\sf cl}( s))\, ^s\gamma({\sf cl}( t))\gamma({\sf cl}( s) \cdot {\sf cl}( t))\1\]
		and thus both $\bar{\theta}$ and $\bar{\theta_1}$ define the same class in $\H^2(\pi_1(\G_*), \pi_2(\G_*))$. Consider now the case when we have chosen another section of ${\sf cl}$, say $\al_1$. Then there exists a function $\eta:\pi_1(\G_*)\to \G_2$ for which
		\[\al_1(z)=\d (\eta(z))\al(z).\]
		As a function $\beta_1$ satisfying the relation $\d(\beta_1(z))=[x,\al_1(x)]$ we can choose
		\[\beta_1(z)=\,^x\eta(z)
		\beta(z)\eta(z)^{-1}.\]
		In fact, we have
		\begin{align*}\d(\beta_1(z))=&\d(^x\eta(z)
			\beta(z)\eta(z)^{-1} )\\
			=&x\d (\eta(z))x^{-1}[x,\al(z)]\d(\eta(z))\1\\
			=&x\d(\eta(z))\al(z)x\1 \al(z)\1 \d(\eta(z))\1\\
			=&x\al_1(z)x\1 \al_1(z)\1\\
			=&[x,\al_1(z)].
		\end{align*}
		Next, we set $a_1=a\eta({\sf cl}( s))^{-1}$. Then we have $s=\d (a_1)\al_1({\sf cl}( s))$. Now we can write
		\begin{align*}\psi_1(s)=&\, ^xa_1\beta_1({\sf cl}(s)) a_1^{-1}\\
			&= \, ^xa_1^x\eta({\sf cl}(s)) \beta({\sf cl}(s))\eta({\sf cl}(s))^{-1} a_1^{-1}\\
			&=
			\, ^x(a_1\eta({\sf cl}( s))) \beta({\sf cl}(s)) (a_1\eta({\sf cl}(s)))^{-1}\\
			&=\psi(s) 
		\end{align*}
		and part vii) is proved.
		
		\item [viii)] Exactness at ${\sf Z}_{\pi_2(\G_*)}(\pi_1(\G_*))$: Take $(x,\xi)\in {\bf Z}_1(G_*)$. Since $\omega'(x,\xi)={\sf cl}(x)$, we can choose $\psi=\xi$ for $g({\sf cl}(x))$. Clearly $\bar{\theta}=1$ for this $\psi$ and hence $g\circ \omega=1$. 
		Take now $x\in \G_1$ such that ${\sf cl}(x)\in {\sf Z}_{\pi_2(\G_*)}(\pi_1(\G_*))$. Assume $g({\sf cl}(x))=1$. Thus there exists a function $\phi: \pi_1(\G_*)\to \pi_2(\G_*)$ for which
		\[\bar{\theta} ({\sf cl}(s),{\sf cl}(t))=\phi({\sf cl}(s))\, ^s\phi({\sf cl}(t))\phi({\sf cl}(st))\1.\]
		Since $\phi$ takes values in $\pi_2(\G_*)$, we see that $\phi({\sf cl}(s))$ is central for all $s\in\G_1$. As $x$ acts trivially on $\pi_2(\G_*)$, it follows that the function $\psi'(t)=\psi(t)\phi(t)$ also satisfies the conditions in i) and ii) and, moreover, $\bar{\theta'}=1$, meaning that $\psi'$ is a $1$-cocycle. Thus 
		$(x,\psi ')\in {\bf Z}_1(\G_*)$ and exactness follows.  
		
	\end{enumerate}
	
\end{proof}

\subsection{Relation to nonabelian cohomology}\label{43se} Crossed modules can be used to define low dimensional non-abelian cohomologies of groups. This was first observed by Dedecker in the 60's (see for instance \cite{dedecker} and references therein) and then developed by D. Guin \cite{dg}, Breen, Borovoi  \cite{borovoi}, Noohi \cite{noohi}. See also \cite{mariam_preparation}.

Let $\G_*$ be a crossed module. According to Borovoi, the group $\H^0(\G_1,\G_*)$ is defined by
$$\H^0(\G_1,\G_*)=\{a\in \G_2|\,\d a=1 \ {\rm and} \  ^xa=a\ {\rm for \ all}\ x\in \G_1\}.$$
Clearly, $\H^0(\G_1,\G_*)=\H^0(\pi_1(\G_*),\pi_2(\G_*))$.
It is a central subgroup of $\G_2$.

In order to define the first cohomology group $\H^1(\G_1,\G_*)$, we first introduce the group $\Der_{\G_1}(\G_1,\G_1)$. Elements of $\Der_{\G_1}(\G_1,\G_1)$ are pairs $(g,\gamma)$ (see \cite[Definition 1.1.]{dg}), where $g\in \G_1$ and $\gamma:\G_1\to \G_2$ is a function for which two conditions hold:
$$\gamma(gh)=\gamma(g) \, ^{g}\gamma(h)\, \ \ \ 
{\rm and}\ \  \d \gamma(t)=[g, t].$$ Here $g,h,t\in\G_1$. Comparing with Definition \ref{z}, we see that these are exactly the conditions (\ref{ZE1}) and (\ref{ZE3}) of Definition \ref{z}. Hence ${\bf Z}_1(\G_*) \subset \Der_{\G_1}(\G_1,\G_2)
$. In fact, ${\bf Z}_1(\G_*)$ is a subgroup of $\Der_{\G_1}(\G_1,\G_2)$, where the group structure on $\Der_{\G_1}(\G_1,\G_2)$ is defined as follows.

If $(g,\gamma), (g',\gamma')\in \Der_{\G_1}(\G_1,\G_1)$ then $(gg',\,  \gamma\ast \gamma')\in \Der_{\G_1}(\G_1,\G_2)$ (see \cite[Lemme 1.2.1]{dg}), 
where $\gamma\ast \gamma'$ is defined by
$$(\gamma\ast \gamma')(t)=\, ^g\gamma'(t)\cdot \gamma(t).$$
Thus $\Der_{\G_1}(\G_1,\G_2)$ is equipped with a binary operation $$(g,\gamma)(g',\gamma')\mapsto (gg',\,  \gamma\ast \gamma').$$ Thanks to  \cite[Lemme 1.2.2]{dg} in this way one obtains a group structure on $\Der_{\G_1}(\G_1,\G_2)$. 

The group $\H^1(\G_1,\G_*)$ is defined as the quotient  $\Der_{\G_1}(\G_1,\G_2)/\sim$ where 
$$(g,\gamma)\sim (g',\gamma')$$
iff there exists $a\in\G_2$ such that $\gamma'(t)=a\1\gamma(t)(^ ta)$ for all $t\in \G_1$ and $g'=\d (a)\1 g$.

The following fact is a direct consequence of the definition.

\begin{Le} One has a commutative diagram with exact rows
$$\xymatrix{
	0\ar[r] &\H^0(\G_1,\G_*)\ar[r]& \G_2\ar[rr]^\dd& &\Der_{\G_1}(\G_1,\G_2)\ar[r] & \H^1(\G_1,\G_*)\ar[r] & 1 \\
	0\ar[r] & \pi_2({\mathcal Z}_*(\G_*)) \ar[u]_{\cong }\ar[r]& \G_2\ar[u]_{Id}\ar[rr]^\dd &&{\bf Z}_1(\G_*)\ar[r] \ar@{^{(}->}[u]
	& \pi_1({\mathcal Z}_*(\G_*)) \ar@{^{(}->}[u] \ar[r]& 1
}
$$

\end{Le}

\subsection{Comparison with Norrie's centre}\label{norr} Let $\G_*$ be a crossed module. Then we have
 $$\H^0(\G_1,\G_2)=\{a\in \G_2| \, ^xa=a\ {\rm for \ all}\ x\in \G_1\}.$$
 Take $a\in \H^0(\G_1,\G_2)$. I claim that $$\d (a)\in {\sf Z}_{\G_2}(\G_1)={\mathcal Z}(\G_1)\cap st_{\G_2}(\G_1).$$ In fact, we have
 $$x\d (a)x\1=\d({^xa})=\d(a)$$
 for all $x\in \G_1$. Hence $\d(a)\in {\mathcal Z}(\G_1)$. For any $b\in \G_2$ we have $\, ^{\d b}a=a$, thus $a\in {\mathcal Z}(\G_2)$. It follows that $\, ^{\d a}b=b$ and hence $$\d (a) \in st_{\G_1}(\G_2)=\{x\in \G_1| ^xa=a \ {\rm for \ all}\ a\in \G_2\}$$ 
 and the claim follows.
 
 Hence we have a homomorphism of abelian groups
 $$\H^0(\G_1,\G_2) \xto{\d} {\sf Z}_{\G_2}(\G_1),$$
 which can be considered as a crossed module with trivial action of the target group on the source. This crossed submodule is denoted by  $\nc(\G_*)$ and is called Norrie's centre of $\G$ \cite[p. 133]{norie}. Take $x\in {\sf Z}_{\G_2}(\G_1)$. Then $(x,{\bf 1})\in {\bf Z}_1(\G_*)$, where ${\bf 1}$ is the constant map $\G_1\to \G_2$ with value $1$. Clearly $j_1(x)=(x,{\bf 1})$ defines an injective homomorphism $j_1:{\sf Z}_{\G_2}(\G_1)\to {\bf Z}_1(\G_*)$. In fact, if we let $j_2$ denote the inclusion $\H^0(\G_1,\G_2)\subset \G_2$, one obtains an injective morphism of crossed modules $j_*:\nc(\G_*) \to {\mathcal Z}_*(\G_*)$ which induces an isomorphism on $\pi_2$ and a monomorphism on $\pi_1$. In general $j_*$ is not a weak equivalence, see Section \ref{ex45}. 
 \subsection{The centre of the crossed module $D_4\to Aut(D_4)$}\label{ex45} 
For any group $G$ there is a crossed module $\partial:G\to Aut(G)$, where $\partial(g)$ is the inner automorphism corresponding to $g\in G$. This crossed modules is denoted by ${\sf AUT}(G)$. The centre of ${\sf AUT}(G)$ is in a sense the ``2-dimensional centre'' of $G$.  We compute this centre for $G=D_4$, the dihedral group of order $8$. Recall that $D_4$ is generated by $a,b$ modulo the relations $a^4=1=b^2$ and $bab=a^3$. Denote by $\bar{D}_4$ the second copy of the same group. To distinguish it from the previous one, we use $\al$ and $\beta$ for the same generators, but now considered as elements of $\bar{D}_4$. Define the homomorphism $$\d:D_4\to \bar{D}_4$$ 
 by $\d(a)=\al^2$, $\d(b)=\beta$. The group $\bar{D}_4$ acts on $D_4$ by
 $$^{\al}a=a,\quad ^{\al}b=ab,$$
 $$^{\beta}a=a\1,\quad ^{\beta}b=b.$$
 Then $\d:D_4\to \bar{D}_4$ is a crossed module and it is easy to check that it is isomorphic to ${\sf AUT}(D_4)$.

 To describe the centre of $\d:D_4\to \bar{D}_4$, we first observe that there are unique crossed homomorphisms $\xi,\eta,\theta:\bar{D}_4\to D_4$ for which
 $$\xi(\al)=a^2, \quad \xi(\beta)=a; \quad \eta(\al)=a, \quad \eta(\beta)=1, \quad \theta(\al)=a^2, \quad \theta(\beta)=1.$$
 Then one checks that the pairs
 $$A=(\al, \xi), \quad B=(\beta,\eta),\quad C=(1,\theta)$$
 belong to ${\bf Z}_1(\d)$. Here, for simplicity, we write ${\bf Z}_1(\d)$ instead of ${\bf Z}_1(\d:D_4\to \bar{D}_4)$.
One easily checks the following  equalities
 \begin{equation}\label{abcrel} C^2=1=B^2=A^4, \quad AC=CA, \quad BC=CB,\quad BAB=A^3.
 \end{equation}
 It follows that in the exact sequence in Lemma \ref{derker} the last map is surjective, because the image contains generators $\al,\beta$ and also it splits. It has the form
 $$0\to {\sf C_2}\to  {\bf Z}_1(\d) \xto{{\sf z}_0} \bar{D}_4 \to 0$$
 where the image of the nontrivial element of ${\sf C_2}$ is $C=(1,\theta)$. It follows that $ {\bf Z}_1(\d)\cong 
 {\sf C_2}\times D_4$. In other words, $ {\bf Z}_1(\d)$  
 is a group generated by $A,B,C$ modulo the relations listed in (\ref{abcrel}). 
  
With these notations, one easily checks that for $\dd:D_4\to {\bf Z}_1(\d)$ one has
$$\dd(a)=A^2 \ {\rm and} \ \dd(b)=BC.$$
Moreover, the corresponding BCM structure on $${\mathcal Z}_*(\d)=\left ( \dd:D_4\to {\bf Z}_1(\d)\right )$$
is in fact a RQM (see Definition \ref{rqm}) and uniquely determined by 
 $$\{\bar{A}, \bar{A}\}=a^2, \quad \{\bar{A}, \bar{B}\}= a, \quad  \{\bar{A}, \bar{C}\}= 1,$$
 $$\{\bar{B}, \bar{A}\}=a, \quad  \{\bar{B}, \bar{B}\}= 1,\quad  \{\bar{B}, \bar{C}\}=1$$
 $$\{\bar{C}, \bar{A}\}=a^2,\quad \{\bar{C}, \bar{B}\}= 1, \quad \{\bar{C}, \bar{C}\}=1.$$
 It follows from Lemma \ref{2le} that the underlying crossed module structure on  $\dd:D_4\to {\bf Z}_1(\d)$ is completely described by $\dd$ and
 $$^Aa=a, \quad   ^Ab=ab, \quad  ^Ba=a\1, \quad ^Bb=b, \quad ^Ca=a, \quad ^Cb=b.$$
 It follows from our description of ${\bf Z}_1(\d)$ and  $\dd$  that $$\pi_i({\mathcal Z _*}(\d))\cong\begin{cases} {\sf C_2}\times {\sf C_2}, \ i=1,\\ {\sf C_2},  \ i=2.\end{cases}.$$
 It is easy to see that
 $$\{1,a^2\}\to \{1\}$$
 is the Norrie's centre of $\d:D_4\to \bar{D}_4$ and hence the inclusion $\nc(\d) \to  {\mathcal Z_*}(\d)$ is not a weak equivalence.
 
 Now we turn to the precrossed module $ {\bf Z}_1(\d) \xto{{\sf z}_0} {\bar D}_4$. Based on part i) of Lemma \ref{10} we obtain
 $$^\al A=A, \quad ^\al B=A^2B,\quad  ^\al C=C, \quad ^\beta A=A\1,\quad  ^\beta B=B, \quad ^\beta C=C.$$
 
 \section{Application in topology}\label{topology}
 \subsection{The Whitehead centre and the main theorem}
 Let ${\mathcal Z}X$  be the centre of a CW-complex $X$ as it is defined in the introduction. We also need the notion of the Whitehead centre of a pointed space $(X,x_0)$, which is denoted by $P(X,x_0)$, see \cite{gottlieb1}. By definition it is the subgroup of elements of  $\pi_1(X,x_0)$ which are central and act trivially on $\pi_k(X,x_0)$ for all $k\geq 2$.

 Recall also that the evaluation at $x_0$ defines the homomorphism $$\pi_1({\mathcal Z}X,\id_X)\to \pi_1(X,x_0),$$
 whose image  is a central subgroup of $\pi_1(X,x_0)$ known as the Gottlieb group  $G(X,x_0)$.  Gottlieb proved that $G(X,x_0)\subset P(X,x_0)$, see \cite[Theorem I.4]{gottlieb1}. 
 For general $X$, the inclusion $G(X,x_0)\subset P(X,x_0)$ is strict, and the computation of $G(X,x_0)$ is an interesting problem.
 
 The following theorem identifies the group \(G(X,x_0)\) with an explicit subgroup of \(P(X,x_0)\) for $X$ a connected CW-complex with $\pi_i(X)=0$ for all $i\geq 3$.
 
 \begin{Th} \label{Got=Wh} Let \({\G}_*=(\G_2\to \G_1)\) be a crossed module such that \(\G_1\) is free and let $X=B{\G}_*$. Then we have an exact sequence
 \[0\to G(X,x_0) \to  {\sf Z}_{\pi_2(\G_*)}(\pi_1(\G_*)) \xto{g} \H^2(\pi_1(\G_*), \pi_2(\G_*)).\]
 Here $P(X,x_0)$ coincides with \({\sf Z}_{\pi_2(\G_*)}(\pi_1(\G_*))\) and \(g\) was constructed in Proposition \ref{15'pr}.
 \end{Th} 
 The proof of Theorem \ref{Got=Wh} is given in Subsection \ref{mainT}, which is based on the technology  of crossed complexes over groupoids \cite{nat}.

\subsection{Crossed complexes over groupoids}
To define crossed complexes we first need to introduce the notion of a crossed module over a groupoid.

\begin{Rem}
	We are now following the convention of \cite{nat}, where actions are defined on the right. This does not cause any problems, as we can translate a right action $a^x$ to an equivalent left action via $^xa:= a^{x\1}$. Thus all the notions and results that we proved for left action have an equivalent formulation using right action.
\end{Rem}

A \emph{crossed module over a groupoid} 
$${\G}_*=({\G}_2 \xto{\partial} {\G}_1\rightrightarrows {\G}_0)$$ consists of a groupoid $({\G}_1\rightrightarrows {\G}_0)$, which will be denoted by $\bar{\G}_*$, together with a covariant functor ${\G}_2:\bar{\G}\to {\sf Groups}$ and  the collection of group  homomorphisms $$\d:{\G}_2(g)\to {\G}_1(g,g).$$ Here $g\in   {\G}_0$ and ${\G}_2(g)$ denotes the value of  ${\G}_2$ on $g$. Moreover, these data must satisfy the following identities:
$$ \d(a^x)=x\1\d(a)x\quad {\rm and} \quad 
a^{\d(b)}=b\1ab. $$
Here $x\in {\sf G}_1(g,h)$, $a,b\in {\G}_2(g)$ and $a^x$ as above  denotes the image of $a$ under the map ${\G}_2(x):{\G}_2(g)\to {\G}_2(h)$.

So when ${\G}_0$ is a one-element set, we recover the usual definition of a crossed module.

If $g,h\in {\G}_0$ and $x,y:x\to y$ are morphisms in $\bar{\G}_*$, then we write $x\sim y$ if $x=y\d_2(z)$ for some  $z\in {\G}_2(g)$. One easily sees that $\sim$ is a congruence and hence we can form the quotient category, which is denoted by   $\pi_1({\G}_*)$. Clearly $\text{Ob}(\pi_1({\G}_*))={\sf G}_0$. As usual with groupoids, $\pi_1({\G}_*,g)$ denotes the group of automorphisms of $g$ in $\pi_1({\G}_*)$. Moreover, $\pi_2({\G}_*,g)$ denotes the group $\ker(\d:{\G}_2(g)\to {\G}_1 (g,g))$.

The set of connected components of 
$\pi_1({\G}_*)$ and $\bar{\G}_*$ are the same, which is denoted by $\pi_0({\G}_*)$.  Varying $g$ we see that the mappings 
$g\mapsto\pi_2({\G}_*,g)$ and $g\mapsto\pi_1({\G}_*,g)$ are  functors $\pi_1(\bar{\G}_*)\to{\sf Ab}$ and $\pi_1(\bar{\G}_*)\to \sf Groups$ respectively. 

A \emph{crossed complex} is an algebraic structure that generalises the notion of crossed modules, extending them to a sequence of higher homotopical objects. Specifically, a crossed complex \( {\mathtt G}_* \) over a groupoid \( \bar{\mathtt G}_*= ({\mathtt G}_1 \rightrightarrows {\mathtt G}_0) \) is defined by a sequence 

\[
\cdots \to {\mathtt G}_4 \xto{\d} {\mathtt G}_3 \xto{\d} {\mathtt G}_2 \xto{\d} {\mathtt G}_1 \rightrightarrows {\mathtt G}_0,
\]
where \( {\mathtt G}_2 \to {\mathtt G}_1 \rightrightarrows {\mathtt G}_0\) is a crossed module over the groupoid \( \bar{\mathtt G}_* \).

For dimensions \( n \geq 3 \), the components \( {\mathtt G}_n \) are functors from \( \bar{\mathtt G}_* \) to the category of abelian groups. Each \( \d: {\mathtt G}_n \to {\mathtt G}_{n-1} \), \( n \geq 3 \) is  a morphism of functors such that  \( \d \circ \d = 0 \). The images of \( {\mathtt G}_2 \) under \( \d \) act trivially on all \( {\mathtt G}_n \) for \( n \geq 3 \), which simplifies their algebraic structure by making \( {\mathtt G}_n \) factor through a simpler groupoid \( \pi_1({\mathtt G}_*) \).

One can define morphisms between crossed complexes that respect this hierarchy, creating the category \( \mathsf{Crs} \) of crossed complexes. Important invariants associated with a crossed complex include:

\begin{itemize}[leftmargin=*]
	\item The \emph{Fundamental Groupoid} \( \pi_1({\mathtt G}_*) \), capturing path components and fundamental group-like properties.
	\item \emph{Higher Homology Groups} \( H_n({\mathtt G}_*, g) = \ker(\d_n) / \operatorname{Im}(\d_{n+1}) \), defined for \( n \geq 2 \).
\end{itemize}

If the object set \( {\mathtt G}_0 \) is a single point, the crossed complex \( {\mathtt G}_* \) is called \emph{reduced}, or a \emph{crossed complex over a group}. In this case, the complex simplifies, with \( \pi_1({\mathtt G}_*) \) becoming a single group rather than a groupoid.

\subsection{Homotopy relations of morphisms in $\sf Crs$}\label{homotcr}
Let $\al_*$ and $\beta_*$ be  morphisms between  crossed complexes ${\mathtt G}_*\to {\mathtt K}_*$. A \emph{homotopy} \cite[Definition 7.1.38, p. 218]{nat} from $\al_*$ to $\beta_*$ is  the following data:
\begin{itemize}[leftmargin=*]
\item a map $\xi_0:{\mathtt G}_0\to {\mathtt K}_1$ such that if
	$g\in {\mathtt G}_0$ then $\xi_0(g)$
	is a morphism from $\al_0(g)$ to $\beta_0(g)$ in the groupoid $\bar{\mathtt K}=({\mathtt K}_1\rightrightarrows {\mathtt K}_0)$; 
	in particular we have $$\al_0(g)=s\xi(g),$$
	where $s:{\mathtt G}_1\to {\mathtt G}_0$ is the source map;
	
	\item a function $\xi_1$ which assigns to each morphism $x:g\to h$ in ${\mathtt G}_1$ an element $\xi_1(x)\in {\mathtt K}_2(\beta_0(h))$ such that the diagram
	$$\xymatrix{\al_0(g)\ar[r]^{\al_1(x)} \ar[d]_{\xi_0(g)}
		& \al_0(h)\ar[d]^{\xi_0(h)}	\\ \beta_0(g)\ar[r]_{\beta_1(x)}& \beta_0(h)}$$
commutes up to $\xi_1(x)$, that is 
$$\al_1(x)=\xi_0(g)\1 \d \xi_1(x)\beta_1(x)\xi_0(g),$$
and one also requires that
$$\xi_1(xy)=\xi_1(x)^{\beta_1(y)} \cdot \xi_1(y)$$
where $x$ and $y$ are composable morphisms in the groupoid $\bar{\mathtt G}_*$;
\item a function $\xi_n$, $n\geq 2$ which assigns to an object $g\in {\mathtt G}_0$ a group homomorphism
$$\xi_n(g):{\mathtt G}_n(g)\to {\mathtt K}_{n+1}(\beta_0(g))$$
such that for any arrow $x:g\to h$ of the groupoid $\bar{\mathtt G}_*$ and any element $a\in {\mathtt G}_n(g), n\geq 2$ one has
$$\xi_n(a^x)=\xi_n(a)^{\beta_1(x)}.$$
Furthermore, one also has 
$$\al_n(a)=\begin{cases} \{\beta_n(a)\cdot \xi_{n-1}\d (a)\cdot \d \xi_n(a)\}^{\xi_0(g)\1}, & n=2\\ \{\beta_n(a)+ \xi_{n-1}\d (a)+\d \xi_n(a)\}^{\xi_0(g)\1},& n>2.\end{cases}
$$
\end{itemize}

Having defined a homotopy, one can talk about homotopy equivalences. As expected, any homotopy equivalence is also a weak equivalence (see \cite[Exercise 7.1.45]{nat}).

%From_the_following To state the result, recall that a crossed complex over a groupoid is called \emph{reduced} if ${\mathtt G}_0$ is a one element set.

Obviously, any groupoid is a disjoint union of its connected components. This implies that any crossed complex over a groupoid 
is also a disjoint union of crossed complexes over connected groupoids. Next, it is also well-known that any connected groupoid is equivalent to a group considered as a one object category. This also has an implication for crossed complexes over groupoids. Having a crossed complex over a groupoid and an object $g$ we can form the following reduced crossed complex:
$$ \quad \cdots \to {\mathtt G}_4(g)\xto{\d_4} {\mathtt G}_3(g)\xto{\d_3} {\mathtt G}_2(g)\xto{\d_2} {\mathtt G}_1(g,g)\to \{g\}.$$
Call it the \emph{fibre over} $g$ and denote it by ${\mathtt G}_*[g]$.
Clearly, $\pi_1({\mathtt G}_*,g)=\pi_1({\mathtt G}_*[g],g)$ and $H_n({\mathtt G}_*,g)= H_n({\mathtt G}_*[g],g)$ for all $n\geq 2$. Thus for connected ${\mathtt G}_*$ the inclusion ${\mathtt G}_*[g]\to G_*$ is a weak equivalence. In fact it is even a homotopy equivalence thanks to \cite[Proposition 7.1.46, pp. 220-221 ]{nat}.

\subsection{Symmetric monoidal closed category structure}

It is an important fact that the category ${\sf Crs}$ has a symmetric monoidal closed category structure \cite[Chapter 9]{nat}, which means that for any crossed complexes ${\mathtt G}_*$ and ${\mathtt K}_*$ there is a ``functional object''
${\sf CRS}_*({\mathtt G}_*,{\mathtt K}_*)$ and a ``tensor object'' ${\mathtt G}_*\t {\mathtt H}_*$ such that
$$Hom_{\sf Crs}({\mathtt H}_*\t {\mathtt G}_*,{\mathtt K}_*)\cong Hom_{\sf Crs }({\mathtt H}_*, {\sf CRS}_*({\mathtt G}_*,{\mathtt K}_*)).$$

We will only use  ${\sf CRS}_*({\mathtt G}_*,{\mathtt K}_*)$. Therefore we recall the corresponding construction following \cite[p. 281]{nat}.

In dimension $0$, ${\sf CRS}_0({\mathtt G}_*,{\mathtt K}_*)$  is the set of crossed complex morphisms from ${\mathtt G}_*$ to ${\mathtt K}_*$. Elements in ${\sf CRS}_1({\mathtt G}_*, {\mathtt K}_*)$ are homotopies. More explicitly, if $\al_*$ and $\beta_*$ are two such morphisms considered as two objects in the crossed complex ${\sf CRS}_*({\mathtt G}_*,{\mathtt K}_*)$,
then the morphisms from $\al_*$ to $\beta_*$ are homotopies $\xi_*$ from $\al_*$ to $\beta_*$ as defined in Section \ref{homotcr}.

One can show that in this way one obtains a groupoid. Next, elements of degree $k$ in ${\sf CRS}_*({\mathtt G}_*, {\mathtt K}_*)$ are $k$-homotopies \cite[Definitions 9.3.3 p. 281]{nat}. Recall that if
$\al_*:{\mathtt G}_*\to {\mathtt K}_*$ is a morphism of crossed complexes as above, then a \emph{$k$-fold homotopy}, $k\geq 2$, from ${\mathtt G}_*$ to ${\mathtt K}_*$ over $\al_*$ is a collection of maps $\xi_n:{\mathtt G}_n\to {\mathtt H}_{n+k}$, $n\geq 0$, satisfying the following conditions:
\begin{itemize}[leftmargin=*]
	\item for $n\geq 2$, $a\in {\mathtt G}_n(g)$ and $x:g\to h$, one has $$\xi_n(a^x)=\xi_n(a)^{\al_1(x)},$$
	
	\item if $n\geq 2$, then $\xi_n$ is additive,
	
	\item for $n=1$, the map $\xi_1$ satisfies the relation
	$$\xi_1(xy)=\xi_1(x)^{\beta_1(y)} \times \xi_1(y).$$
\end{itemize}

Thus, for each morphism  $\al_*: {\mathtt G}_*\to {\mathtt K}_*$ considered as an object of the groupoid ${\sf CRS}({\mathtt G}_*,{\mathtt K}_*)$, one denotes by ${\sf CRS}_k({\mathtt G}_*,{\mathtt K}_*)(\al_*)$ the collection of all $k$-fold homotopies over $\al_*$. Then the assignment
$$\al_*\mapsto  {\sf CRS}_k({\mathtt G}_*,{\mathtt K}_*)(\al_*)$$ 
is a part of the crossed complex ${\sf CRS}_*({\mathtt G}_*,{\mathtt K}_*)$, see details in \cite[Definition 9.3.5, p. 282]{nat}.

We will need the following fact, which is a straightforward consequence of the definition.

Let $n\geq 2$ and ${\mathtt G}_*$ a crossed complex over a groupoid. We write $$\ell({\mathtt G}_*)\leq n$$ 
if ${\mathtt G}_k=0$ for all $k>n$. If ${\mathtt G}_*$ is a crossed complex over a groupoid, we let $\tau_n{\mathtt G}_*$ be the following quotient of ${\mathtt G}_*$:
$$\tau_m{\mathtt G}_*=\begin{cases} {\mathtt G}_m, & m<n\\ {\mathtt G}_n/\im(\d),& m=n\\ 0&m>n
\end{cases}
$$
\begin{Le}\label{crstau} Let ${\mathtt G}_*$ and ${\mathtt K}_*$ be crossed complexes over groupoids. Assume $\ell({\mathtt K}_*)\leq n$, where $n\geq 2$. Then
	$${\sf CRS}_k({\mathtt G}_*,{\mathtt K}_*)={\sf CRS}_k(\tau_n{\mathtt G}_*,{\mathtt K}_*).$$
\end{Le}

The situation when both ${\G}_*$ and ${\K}_*$ are crossed modules is especially transparent. In this case we have ${\G}_0=\{1\}={\K}_0$ and ${\G}_n=0={\K}_n$ for all $n\geq 3$. Hence  ${\sf CRS}_*({\G}_*,{\K}_*)$ has the form
$$\cdots \to 0\to {\sf CRS}_2({\G}_*,{\K}_*)\to {\sf CRS}_1({\G}_*,{\K}_*)\rightrightarrows {\sf CRS}_0({\G}_*,{\K}_*)$$
where $${\sf CRS}_0({\G}_*,{\K}_*)=Hom_{\sf Crs}({\G}_*, {\K}_*).$$
Therefore objects are given by $\al_*=(\al_2,\al_1)$, where  $\al_{2}:{\G}_{2}\to {\K}_{2}$ and $\al_1:{\G}_{1}\to {\K}_{1}$ are group homomorphisms such that the diagram 
$$\xymatrix{\cdots\ar[r] &0\ar[r]\ar[d]_\id  &{\G}_{2}\ar[r]^\d\ar[d]_{\al_{2}}& {\G}_{1}  \ar[d]^{\al_1}\ar[r] & \{1\}\ar[d]^\id\\ \cdots\ar[r]  &0\ar[r]  &{\K}_{2}\ar[r]_{\d} &{\K}_{1}\ar[r] & \{1\}} $$
commutes and 
$$\al_{2}(^xa )=\,  ^{\al_{1} (x)}\al_{1}(a).
$$
If $\al=(\al_1,\al_2)$ and $\beta=(\beta_2,\beta_1)$ are two morphisms ${\G}_*\to {\K}_*$, then a homotopy $\al \Longrightarrow \beta$ is a pair $\xi_*=(\xi_0,\xi_1)$, where $\xi_0:{\G}_0\to {\K}_1$ and $\xi_1:{\G}_1\to {\K}_2$ are maps satisfying the following properties (since ${\G}_0=\{1\}$ we will assume that $\xi_0\in {\K}_1$).
$$\xi_1(xy)=\xi_1(x) ^{\beta_1(x)}\, \xi_1(y),$$
$$\al_1(x)=\xi_0\1\d(\xi_1(x))\beta_1(x)\xi_0.$$
$$\al_2(a)=\, (\beta_2(a)\xi_1(\d a))^{\xi_0(g)\1}.$$

Here $x,y\in {\G}_1$ and $a\in {\G}_2$. Thus ${\sf CRS}_1({\G}_*,\K_*)(\al,\beta)$ is the set of all pairs $(\xi_0,\xi_1)$ satisfying the above three conditions. Finally, ${\sf CRS}_2({\G}_*,{\K}_*)(\al)$ is the set of all maps $\tau:{\G}_0\to {\K}_2$. Since ${\G}_0=\{1 \}$, we see that ${\sf CRS}_2({\G}_*,{\K}_*)={\G}_2$. 

\begin{Le} Let ${\G}_*$ be a crossed module. Then the fibre of ${\sf CRS}_*({\G}_*,{\G}_*)$ at $id=(\id_{{\G}_2},\id_{{\G}_1})$ is equivalent to ${\mathcal Z}_*({\G}_*)$.
\end{Le}
\begin{proof} If $\al=\beta=\id$, then we see that the fibre in dimension 1 consists of pairs $(h_0,h_1)$, where $h_0\in \G_1$. These pairs must satisfy the conditions
	\begin{align} 
		h_1(xy)&=h_1(x) ^y\cdot h_1(y),\\
		x&=h_0\1 \d (h_1(x)) xh_0,\\
		a&=\, (a h_1(\d a))^{(h_0\1)}.
	\end{align}
	One easily sees that the first equality simply says that $h_1:\G_1\to \G_2$ is a crossed homomorphism, the second one says that $\d(h_1(x))=h_0xh_0\1x\1=[h_0,x]$, while the third one says $h_1(\d a)=a\1 a^{h_0}$ and the result follows.
	
\end{proof}

\subsection{Relation to mapping spaces}
The category ${\sf Crs}$ of crossed complexes   and the category ${\sf CW}$ of CW-complexes and cellular maps are related by the pair of  functors $$\Pi_*:{\sf CW}\to {\sf Crs}\quad  {\rm and} \quad  B:{\sf Crs}\to {\sf CW},$$ such that on homotopy classes of maps one has a binatural bijection:
$$[X,B\G_*]\cong [\Pi_* X, \G_* ].$$
Actually more is true, see \cite[Theorem 11.4.19. p.378]{nat}. 

\begin{Th}\label{BHTHA}
	Let $X$ be a $CW$-complex and let
	$\G_*$ be a crossed complex.
	Then there is a natural weak homotopy equivalence
	$$ B({\sf CRS}_*(\Pi_* X,\G_*)) \to Maps(X, B\G_*). $$
\end{Th}

The proof of Theorem \ref{Got=Wh} is based on the following result, which is of independent interest. 

\begin{Pro}\label{709.pr2} If $\G_*=(\G_2\xto{\d} \G_1)$ is a crossed module 
	such that $\G_1$ is a free  group, 
	then $B{\mathcal Z}\G_*$ and ${\mathcal Z}B\G_*$ are homotopy equivalent.
\end{Pro}

\begin{proof}

Take $X=B\G_*$ in Theorem \ref{BHTHA} to obtain a weak equivalence 
$$B({\sf CRS}_*(\Pi_* B\G_*,\G_*))\to Maps(X, B\G_*).$$
By Lemma \ref{crstau} we have that the first space can be replaced by
$B({\sf CRS}_*(\tau_2\Pi_* B\G_*,\G_*))$. It is well-known that  $\Pi_*B\G_*\to  \G_*$ is a weak equivalence (see for example, \cite[p.100]{BH2}). It follows that $\tau_2(\Pi_*B\G_*)\to \G_*$ is also a weak equivalence. Observe that $\tau_2(\Pi_*B\G_*)$ is a crossed module with free group in dimension one, same for $\G_*$. It follows that these crossed modules are actually weak equivalent and cofibrant object in an appropriate model category structure, see \cite[Corollary 2.13] {miranda}. Hence they are homotopy equivalent and as a consequence $B({\sf CRS}_*(\tau_2\Pi_* B\G_*,\G_*))$ and $B({\sf CRS}_*(\G_*,\G_*))$ are homotopy equivalent. Hence we obtain the weak equivalences
$$B({\sf CRS}_*(\G_*,\G_*))\to B({\sf CRS}_*(\Pi_* B\G_*,\G_*))\to Maps(X, B\G_*).$$

Looking at the component containing the identity map, we obtain a homotopy equivalence $$ B{\mathcal Z}_*\G_*\to {\mathcal Z}B\G_*,$$
proving the proposition.
\end{proof}
\subsection{Proof of Theorem \ref{Got=Wh}}
\label{mainT}
Let \(\G_*=(\G_2\xto{\d} \G_1)\) be a crossed module 
such that $\G_1$ is a free group. Consider the evaluation map
\[{\mathcal Z}X=B{\mathcal Z}_*\G_*\to B\G_* = X\]
whose algebraic model is the crossed module morphism \((id_{\G_2}, {\sf z}_1):{\mathcal Z}_*(\G_*)\to \G_*\). Hence, the homomorphism between the homotopy groups
\[\xymatrix{ \pi_1({\mathcal Z}X,\id_X)\ar[r]\ar[d]_{id}& \pi_1(X,x_0)\ar[d]^{id} \\
	\pi_1({\mathcal Z}_*(\G_*))\ar[r]& \pi_1(\G_*)
}\]
is induced by \( {\sf z}_1:(x,\xi) \mapsto x\) and thus \(\pi_1({\mathcal Z}X,\id_X)\to \pi_1(X,x_0)\) factors through the map \(\omega\) from Proposition \ref{15pr}. Hence,
\begin{align*}
G(X,x_0) & = \im (\pi_1({\mathcal Z}X,\id_X)\to \pi_1(X,x_0) )\\
& = \im (\pi_1({\mathcal Z}_*(\G_*))\xto{\omega}  {\sf Z}_{\pi_2(\G_*)}(\pi_1(\G_*)))\\
& = \ker (g),
\end{align*}
where \(g\) is the homomorphism defined in Proposition \ref{15'pr}.

%Since the map $ G(X,x_0)\to P(X,x_0)$ is injective \cite{gottlieb1} we need to show that it is also surjective. 
%Without loss of generality we can assume that $X=B\G_*$, where $\G_*=(\G_2\to \G_1)$ is a crossed module with free $\G_1$. In this case, $P(X,x_0)={\sf Z}_{\pi_2(\G_*)}(\pi_1(\G_*))$ (see Theorem \ref{15pr}) and the result follows from part iv) of the same Theorem, because by our assumption  $\G_1$ is a free group and hence $\H^2(G_1,\pi_2(\G_*))=0$.

\section*{Acknowledgements}
This research was partially supported by the grant of the Shota Rustaveli Georgian National Science Foundation FR-22-199.

\end{document}